\documentclass[11pt,reqno]{amsart}

\usepackage[margin=1.5in]{geometry}

\usepackage{style}

\let\cong\equiv\relax
\renewcommand{\equiv}{\simeq}

\newcommand{\pcpl}{_p^\compl}
\newcommand{\pIcpl}{_{(p,I)}^\compl}

\newcommand{\pqmocpl}{_{(p,q-1)}^\compl}

\newcommand{\RGammagr}{\RGamma^\gr}

\newcommand{\Syn}{\mathrm{Syn}}

\DeclareMathName{\Fun}{Fun}
\DeclareMathName{\Shv}{Shv}
\DeclareMathName{\Sp}{Sp}
\DeclareMathName{\CAlg}{CAlg}
\DeclareMathName{\DAlg}{DAlg}
\DeclareMathName{\QSyn}{QSyn}
\DeclareMathName{\Mod}{Mod}
\DeclareMathName{\QCoh}{QCoh}
\DeclareMathName{\Comod}{Comod}
\DeclareMathName{\Stable}{Stable}
\DeclareMathName{\Stk}{Stk}
\newcommand{\Modgr}[1]{\Mod_{#1}(\calD(\BGm))}

\newcommand{\Cathat}{\widehat{\cat{Cat}}}

\newcommand{\DF}{\mathcal{DF}}
\newcommand{\DFcpl}{\widehat{\DF}}

\newcommand{\SpecBGm}{\Spec_{\BGm}}
\newcommand{\SpfBGm}{\Spf_{\BGm}}

\newcommand{\Einfty}{{\bbE_\infty}}

\DeclareMathName{\KK}{K}
\DeclareMathName{\THH}{THH}
\DeclareMathName{\TC}{TC}
\newcommand{\TCm}{\TC^-}
\newcommand{\TPtw}{\TP^{(-1)}}
\DeclareMathName{\TP}{TP}
\DeclareMathName{\HH}{HH}
\DeclareMathName{\HP}{HP}
\DeclareMathName{\HC}{HC}

\DeclareMathName{\MU}{MU}
\DeclareMathName{\KU}{KU}
\DeclareMathName{\ku}{ku}

\newcommand{\Mfg}{\mathcal{M}_\mathrm{fg}}

\newcommand{\Laz}{\mathcal{L}}

\newcommand{\Prism}{\mathbbl{\Delta}}
\newcommand{\Prismhat}{\widehat{\Prism}}
\newcommand{\Prismbar}{\overline{\Prism}}

\newcommand{\vect}{\mathrm{vect}}

\DeclareMathName{\gauge}{Gauge}
\newcommand{\gaugecpl}{\widehat{\strut\gauge}}
\newcommand{\gaugedcpl}{\gaugecpl^\mathrm{dec}}
\newcommand{\Fgauge}{F\mathhyphen\gauge}
\newcommand{\Fgaugecpl}{\widehat{\Fgauge}}
\newcommand{\Fgaugevect}{\Fgauge^\vect}
\DeclareMathName{\Crys}{Crys}

\newcommand{\Fcrysvect}{F\mathhyphen\Crys^\vect}

\newcommand{\hatG}{\widehat{G}}
\newcommand{\GQ}{\hatG^\mathrm{Q}}
\newcommand{\GDr}{\hatG^\mathrm{Dr}}
\newcommand{\GDralg}{G^\mathrm{Dr}}
\newcommand{\aff}{\mathrm{aff}}
\newcommand{\GDraff}{\hatG^\mathrm{Dr,aff}}
\newcommand{\GDraffalg}{G^\mathrm{Dr,aff}}
\newcommand{\Guniv}{\hatG^\mathrm{univ}}

\newcommand{\Nyg}{\mathcal{N}}
\newcommand{\FilN}{\Fil_\Nyg}
\newcommand{\grN}{\gr_\Nyg}

\newcommand{\rmconj}{\mathrm{conj}}
\newcommand{\Filconj}{\Fil^\rmconj}

\DeclareMathName{\mot}{mot}
\newcommand{\Filmot}{\Fil_\mot}
\newcommand{\grmot}{\gr_\mot}

\DeclareMathName{\Perf}{Perf}
\newcommand{\Gm}{\bbG_m}

\newcommand{\BGm}{\bfB\bbG_m}

\newcommand{\bbCP}{\mathbb{CP}}
\newcommand{\CPinfty}{{\bbCP^\infty}}

\newcommand{\Zp}{\bbZ_p}
\newcommand{\Qp}{\bbQ_p}
\newcommand{\Cp}{\bbC_p}
\newcommand{\OCp}{\calO_{\Cp}}
\newcommand{\Cpflat}{\Cp^\flat}

\newcommand{\QdR}{{\Zp\Brk{q-1}}}

\newcommand{\QdRcyc}{{\Zp[q^{1/p^\infty}]\pqmocpl}}

\newcommand{\Zpcyc}{{\Zp^{\mathrm{cyc}}}}

\newcommand{\Zptimes}{{\Zp^\times}}

\newcommand{\chicyc}{\chi_\mathrm{cyc}}

\newcommand{\stilde}{\tilde{\hspace{0.06ex}s}}

\title{On the Drinfeld formal group}
\author{Deven Manam}

\begin{document}

\begin{abstract}
We identify Drinfeld's formal group on the prismatization of $\Spf \Zp$ with a formal group arising from homotopy theory, given locally by the Quillen formal group of a decompleted variant of topological periodic cyclic homology.
We also prove that this formal group extends to the syntomification of $\Spf \Zp$, and that it admits a certain algebraization conjectured by Drinfeld.
\end{abstract}

\maketitle

\tableofcontents

\section*{Introduction}

Prismatic cohomology, developed by Bhatt--Morrow--Scholze and Bhatt--Scholze, is a $p$-adic cohomology theory for $p$-adic formal schemes which recovers and unifies most previously known cohomology theories of this type \cite{bms1,bms2,prisms}.
Work of Bhatt--Lurie and Drinfeld elucidates and organizes many aspects of this theory through a ``stacky'' reformulation \cite{apc,blprismatization,fgauges,prismatization}.
This approach associates to any $p$-adic formal scheme $X$ certain (derived) stacks $X^\Prism$, $X^\Nyg$, and $X^\Syn$, sheaves on which provide a good category of coefficients for prismatic, Nygaard-filtered prismatic, and syntomic cohomology of $X$, respectively.

In \cite{drinfeldfg}, Drinfeld introduces a certain formal group over the stack $\Zp^\Prism$, which we denote $\GDr$.
This formal group can be used to construct the prismatic logarithm of \cite[\S 2]{apc} (and the $q$-logarithm of \cite[\S 4]{albtrace}) and is thus related to the theory of Chern classes in prismatic cohomology.

On the other hand, the approach of \cite{bms2} to prismatic cohomology via topological Hochschild homology shows that, for any quasiregular semiperfectoid (qrsp) ring $R$, the ring spectrum $\TP(R; \Zp)$ is even and $2$-periodic, with its $0$\textsuperscript{th} homotopy group recovering $\Prismhat_R$, the Nygaard-completed prismatic cohomology of $R$.
Thus the yoga of complex-orientable cohomology theories supplies us with a Quillen formal group $\Spf \pi_0 \prn*{\TP(R; \Zp)^\CPinfty}$ over $\Spf \Prismhat_R$.
In fact, using the results of \Cref{sec:cplFgauges} we may extend this to define a formal group $\GQ$ over $\Zp^\Syn$, which we (in a minor abuse of terminology) also call the Quillen formal group.
One naturally wonders about the relation between these two formal groups, which leads us to our first theorem.

\begin{theorem}[\Cref{drinfeldquillenfg}, \Cref{drinfeldfgZpSyn}]\label{drinfeldquillenfgintro}
The Quillen formal group is an extension of $\GDr$ along the open immersion $\Zp^\Prism \inj \Zp^\Syn$, uniquely determined by the identification of its cotangent space at the identity with the Breuil--Kisin line bundle $\calO_{\Zp^\Syn}\{1\}$.

As such, for a qrsp ring $R$, there is a natural isomorphism of formal groups
\[ \GDr_{\Prism_R} \isom \Spf \pi_0 \prn*{\TPtw(R; \Zp)^\CPinfty} \mpunc{\raisebox{-1.1ex}{,}} \]
where $\TPtw$ is the Nygaard-decompleted Frobenius-untwist of $\TP$ defined in \Cref{sec:TPtw}.
\end{theorem}

\begin{remark}
This \namecref{drinfeldquillenfgintro} generalizes results of Morava on certain complex orientations of $\THH(\calO_K; \Zp)$ for perfectoid fields $K$ \cite{moravathh}, and of Angelini-Knoll on complex orientations of $\TC^-(\calO_L / \bbS[z]; \Zp)$ for $p$-adic local fields $L$ \cite{orientationoftp}.
\end{remark}

We now turn to the question of finding an \emph{algebraization} of the formal group $\GDr$ in the sense of \cite[\S 2.12.1]{drinfeldfg}, i.e.\ of showing that it arises as the completion at the identity section of a (relative) smooth affine group scheme.
Drinfeld shows that the pullbacks of $\GDr$ to the $q$-de Rham prism and along the Frobenius morphism of $\Zp^\Prism$ admit compatible algebraizations, and he conjectures that these algebraizations arise from an algebraization of $\GDr$ itself.
He also observes that any such algebraization is necessarily unique.
Our second main theorem resolves this conjecture.

\begin{theorem}[\Cref{drinfeldfgalg}, {\cite[Conjecture 2.12.4]{drinfeldfg}}]\label{drinfeldfgalgintro}
The formal group $\GDr$ (along with its extension to $\Zp^\Syn$) admits an algebraization $\GDralg$, given after pullback to the $q$-de Rham prism $\QdR$ by
\[ \Spf \QdR[z]\pqmocpl \]
with group structure
\[ z \mapsto z_1 + z_2 + (q-1) z_1 z_2 \mpunc. \]
\end{theorem}

Our proof of this result crucially uses the extension of $\GDr$ to a formal group over $\Zp^\Syn$, but it is otherwise essentially independent of \Cref{drinfeldquillenfgintro}.

\begin{remark}
Drinfeld describes a certain section $\stilde_{\QdR}$ of the $q$-de Rham pullback of this algebraic group, given under the description above by sending $z \mapsto 1$.
By \Cref{stildedescent}, this section descends to a section $\stilde$ of $\GDralg$.

Note that the pullback of the $n$-series of the formal group law above along $\stilde_{\QdR}$ is the $q$-analogue $[n]_q$.
As such, we may in some sense regard the section $\stilde$ as providing an absolute prismatic descent of $q$-analogues, via the composites
\[ \Zp^\Prism \xrightarrow{\stilde} \GDralg \xrightarrow{n} \GDralg \mpunc. \]
\end{remark}

\begin{remark}
Lurie has described in \cite{f1syn} a different proof of this conjecture using the stack $\bbF_1^\Syn$, which, like ours, also yields an algebraization over $\Zp^\Syn$.
Over $\Zp^\Prism$ the two algebraizations necessarily coincide, and we expect that they coincide over all of $\Zp^\Syn$.%
\footnote{It seems likely that one can show this using \Cref{drinfeldfgZpSyn}, \cite[\S 2.12.2]{drinfeldfg}, and the explicit description of our algebraization over $\OCp^\Syn$ implicit in the proof of \Cref{drinfeldfgalg}.}

We note that, using some ideas from \cite{f1syn}, one can go in the opposite direction to recover the $v_1$-completed stack $\widehat{\bbF}_1^\Syn$ from the existence of the algebraization (although we do not do so here).
As such, we believe that the algebraization given here, which is defined in much greater generality, will be useful for understanding related ideas in the setting of prismatization of ring spectra, as studied in \cite{dhry}.
\end{remark}

In the course of proving \Cref{drinfeldfgalgintro}, we make the following observation about the moduli stack of formal groups.

\begin{proposition}[\Cref{Mfgaffine}, \Cref{affdiagaffineness}]
Let $\calO(1) \in \calD(\Mfg)$ denote the the cotangent space of the universal formal group at the identity section, and write $\RGammagr(\Mfg) \defeq \RGamma(\Mfg, \calO(*))$ for the global sections of all its powers, regarded as a graded derived ring.
Then the natural map of stacks
\[ \Mfg \to \SpecBGm \RGammagr(\Mfg) \mpunc, \]
where the target is defined as in \Cref{gradedderivedrings}, is an equivalence.

In other words, $\Mfg$ is a relative affine stack over $\BGm$, for the map $\Mfg \to \BGm$ classifying the cotangent space of the universal formal group.
\end{proposition}

The result above is in some sense well-known, and related statements appear in work on Hovey's stable module category \cite{hoveystable,bhv} and on the ``cofiber of $\tau$'' philosophy \cite{cofibtau,synthetic,Cmmf,mfgorconn}.
However, we expect that the formulation given here will be of independent utility.

\subsection*{Acknowledgments}
I am grateful to Greg Andreychev, Ben Antieau, Sanath Devalapurkar, Rok Gregoric, Adam Holeman, Dhilan Lahoti, Akhil Mathew, Arpon Raksit, and Noah Riggenbach for many interesting conversations related to this work.
I am also grateful to Bhargav Bhatt, Elden Elmanto, Jeremy Hahn, and especially Ben Antieau for helpful comments on a draft.

I would also like to thank Jacob Lurie for explaining an alternative proof of \Cref{MfgQCohaffine} (see \Cref{jacobaffineness}).

During the writing of this paper, the author was supported in part by NSF grant DMS-2102010.

\subsection*{Organization}
In \Cref{sec:TPtw} we define the variant $\TPtw$ of $\TP$.

In \Cref{sec:cplFgauges} we study a variant of prismatic $F$-gauges having to do with Nygaard-completed prismatic cohomology.

In \Cref{sec:gradedstacks} we recall some facts about stacks over $\BGm$ and prove relative affineness of $\Mfg$.

In \Cref{sec:drinfeldquillenfg}, which is mostly independent of \Cref{sec:gradedstacks}, we construct the Quillen formal group and show that it extends the Drinfeld formal group.

In \Cref{sec:algebraization} we construct the conjectured algebraization of the Drinfeld formal group.

\subsection*{Conventions}

We let $\Cp$ denote the completion of an algebraic closure of $\Qp$, and we fix a compatible system of $p$-power roots of unity $\epsilon \in T_p \Cp^\times$.
We define as in \cite[\S 3.2]{bms1} the elements $\mu \defeq [\epsilon] - 1$, $\xi \defeq \frac{\mu}{\phi^{-1}(\mu)}$, and $\tilde\xi \defeq \phi(\xi)$ of $\Prism_{\OCp}$.
Note that this induces a map of prisms
\[ (\QdRcyc, ([p]_q)) \to (\Prism_{\OCp}, (\tilde\xi)) \]
sending $q - 1 \mapsto \mu$.

When working with a prism $A$, we will often write $I$ for the Hodge--Tate ideal of $A$, leaving the dependence on $A$ implicit.

We abbreviate ``quasiregular semiperfectoid'' to ``qrsp''.

All formal groups are assumed to be $1$-dimensional and commutative.
For brevity, we refer to the cotangent space of a formal group at the identity section simply as the cotangent space of the formal group.

Our gradings and filtrations are all $\bbZ$-indexed unless specified otherwise.

We make use of $\infty$-categories through the theory of quasicategories developed by Boardman--Vogt, Joyal, and Lurie.
Our primary references for this theory are \cite{htt} and \cite{ha}.

Our categories are all $\infty$-categories, and our stacks are all fpqc $\infty$-stacks.
However, our schemes are ordinary schemes, and our rings by default ordinary commutative rings.

In order to reduce ambiguity when working with both t-structures and gradings, we denote the connective and coconnective parts of a t-structure on a stable category $\calC$ by $\tau_{\geq 0} \calC$ and $\tau_{\leq 0} \calC$ respectively.
As we have no need to denote the truncation of $\calC$ itself, the conflict of notation does not present a problem for us.

Given a stack $X$, we define the derived category $\calD(X)$ by right Kan extension from affine schemes, as in \cite{dag}.
This category carries a $t$-structure whose connective objects are those that are connective after pullback to an arbitrary affine scheme.
Note that, for this category to be well-defined in complete generality, we must either fix a cutoff cardinal on our fpqc site or regard $\calD(X)$ as generally living only in a larger universe.
We opt for the latter, implicitly noting that, for the stacks of actual interest to us, the derived category is well-defined in the smaller universe.

\section{A variant of \texorpdfstring{$\TP$}{TP}}\label{sec:TPtw}

\begin{recollection}[{\cite[\S 5.7]{apc}}, {\cite[\S 7.4]{apc}}]
For any animated commutative ring $R$ we have divided Frobenius maps
\[ \phi_i : \FilN^i \Prism_R \{i\} \to \Prism_R \{i\} \mpunc, \]
which are compatible with the Nygaard filtration on the source and the $I$-adic filtration on the target.
Since the $I$-adic filtration is complete, these maps factor canonically through the Nygaard completion:
\[ \FilN^i \Prism_R \{i\} \to \FilN^i \Prismhat_R \{i\} \xrightarrow{\tilde\phi_i} \Prism_R \{i\} \mpunc. \]
\end{recollection}

We need the following variant of \cite[Corollary 5.2.16]{apc} and \cite[\S 5.5.1, 3(d)]{fgauges}; the proof follows the sketch given in \emph{loc.\ cit.}.

\begin{proposition}\label{divFroblocalization}
For any qrsp ring $R$, the divided Frobenius map
\[ \tilde\phi : \bigoplus_i \FilN^i \Prismhat_R \{i\} t^{-i} \to \bigoplus_i \Prism_R \{i\} t^{-i} \]
is a graded $(p, I)$-complete localization.
\end{proposition} \begin{proof}
Since the Breuil--Kisin twist over $R$ is trivializable, this is equivalent to showing the same statement for the (non-divided) Frobenius map
\[ \bigoplus_i \FilN^i \Prismhat_R t^{-i} \to \bigoplus_i I^i \Prism_R t^{-i} \mpunc, \]
which, as a result of \cite[Construction 2.2.14]{apc}, is obtained from $\tilde\phi$ by applying the Breuil--Kisin twist by $-i$ in each weight $i$.
Choose a perfectoid ring $R_0$ mapping to $R$, and let $e \in \FilN^1 \Prismhat_{R_0}$ be a generator, so that $\phi(e)$ is a generator of $I$.
As the image of $e t^{-1}$ is invertible we obtain a map
\[ \prn*{\bigoplus_i \FilN^i \Prismhat_R t^{-i}}\brk*{\frac{1}{e t^{-1}}}\pIcpl \to \bigoplus_i I^i \Prism_R t^{-i} \mpunc, \]
which we claim is an isomorphism.
The $t$-adic filtration is complete on both sides above: for the target this follows from completeness of the $I$-adic filtration, and for the source this follows from the fact that $e = t \cdot e t^{-1}$ satisfies $e^p \in (p, I)$.
It therefore suffices to prove the claim modulo $t$.
Unwinding everything, we need to show that for each $i$ the map
\[ \colim_e \grN^{i+\bullet} \Prismhat_R \to \Prismbar_R\{i\} \]
induced by the Frobenius is a $p$-complete equivalence.
Here the map
\[ \grN^{i+j} \Prismhat_R \to \Prismbar_R\{i\} \]
is given by taking the Frobenius, which lands in $\Prismbar_R\{i+j\}$, and multiplying by $\phi(e)^{-j}$.
Thus by \cite[Theorem 12.2]{prisms} it is injective, with image agreeing with $\Filconj_i \Prismbar_R\{i\}$ up to multiplication by a unit.
As this unit comes from $\Prismbar_{R_0}$, it lies in $\Filconj_0 \Prismbar_R$, so the image is in fact exactly $\Filconj_i \Prismbar_R\{i\}$.
Now exhaustivity of the conjugate filtration \cite[Construction 7.6]{prisms} allows us to conclude.
\end{proof}

In light of the above proposition, we define, for any qrsp $R$, the subset
\[ \scrS_R \subseteq \bigoplus_i \FilN^i \Prismhat_R \{i\} t^{-i} \]
to be the set of homogeneous elements with invertible image under $\tilde\phi$.
We may now make the following definition, lifting the previous \namecref{divFroblocalization} to spectra.

\begin{definition}
For $R$ qrsp, \cite[Theorem 1.12]{bms2} and \cite[Theorem 13.1]{prisms} yield an identification
\[ \pi_{2i} \TCm(R; \Zp) \isom \FilN^i \Prismhat_R \{i\} \mpunc. \]
Thus, using \Cref{loccomplfamily}, we may define a commutative ring spectrum
\[ \TPtw(R; \Zp) \defeq \prn*{\scrS_R^{-1} \TCm(R; \Zp)}\pIcpl \]
functorially in $R$.%
\footnote{To be precise, we perform this construction over the category of $\kappa$-small qrsp rings for each cardinal $\kappa$, then pass to the colimit.}
This spectrum has vanishing odd homotopy groups, and by \Cref{divFroblocalization} its even homotopy groups satisfy
\[ \pi_{2i} \TPtw(R; \Zp) \isom \Prism_R\{i\} \mpunc. \]
We also define the motivic filtration
\[ \Filmot^\bullet \TPtw(R; \Zp) \defeq \tau_{\geq 2\bullet} \TPtw(R; \Zp) \mpunc. \]
As the associated graded of the motivic filtration satisfies quasisyntomic descent, so does $\TPtw$, so we may extend this construction by descent to all quasisyntomic rings.
We thus obtain a sheaf of complete filtered commutative ring spectra
\[ \Filmot^\bullet \TPtw(-; \Zp) \]
on the quasisyntomic site such that
\[ \grmot^* \TPtw(R; \Zp) \equiv \Prism_R\{*\} \mpunc. \]
\end{definition}

\begin{remark}
The notation $\TPtw$ reflects the fact that this object should be viewed not just as a decompletion of $\TP$ but also as a Frobenius untwist.
This is clearer in the relative setting, where the motivic filtration on $\TP$ only recovers the Nygaard-completed Frobenius twist of prismatic cohomology, while performing the analogue of the construction above recovers prismatic cohomology on the nose.
Indeed, this construction is essentially a coordinate-free version of the ``Frobenius descent'' of \cite[\S 11.3]{bms2}.
\end{remark}

\begin{remark}
There is by construction a canonical filtered factorization
\[ \TCm(-; \Zp) \to \TPtw(-; \Zp) \to \TP(-; \Zp) \]
of the Frobenius map.
The second map above exhibits the graded pieces of the target as the Nygaard-completions of those of the source.
\end{remark}

\begin{remark}\label{TPtwdecompleted}
There is a variant of the above construction defined for all animated rings, which one can define by taking the completion of the tensor product
\[ \TCm(-; \Zp) \tensor_{\TCm(\Zp; \Zp)} \TP(\Zp; \Zp) \]
with respect to a certain secondary filtration, or by taking the same completion of the left Kan extension of $\TP(-; \Zp)$ from polynomial rings to all animated rings.
The motivic filtration on this variant is not necessarily complete, even on quasisyntomic rings.
This object is studied in forthcoming work of Devalapurkar, Hahn, Raksit, and Yuan \cite{dhry}.
\end{remark}

\section{Completed \texorpdfstring{$F$}{F}-gauges}\label{sec:cplFgauges}

\DeclareMathName{\fndec}{dec}
\DeclareMathName{\fncpl}{cpl}

We first recall the definition of a prismatic gauge and of a prismatic $F$-gauge.

\begin{definition}[{\cite[Example 6.1.7]{fgauges}}, {\cite[Definition 2.14]{tangsyntomic}}]
For $R$ qrsp, we define the category of \emph{prismatic gauges} over $R$ to be the $(p, I)$-complete filtered derived category
\[ \gauge(R) \defeq \DF(\FilN^\bullet \Prism_R)\pIcpl \]
of the filtered ring $\FilN^\bullet \Prism_R$.%
\footnote{
  Recall that the filtered derived category of a filtered ring $\Fil^\bullet A$ is the category of modules over $\Fil^\bullet A$ in the category of filtered complexes $\DF(\bbZ) \defeq \Fun(\bbZ_{\geq}, \calD(\bbZ))$.
  For an ideal $I \subseteq A$, the $I$-complete filtered derived category consists of those objects of $\DF(\Fil^\bullet A)$ each of whose filtered pieces is $I$-complete.
}

A \emph{prismatic $F$-gauge} over $R$ consists of a prismatic gauge $\Fil^\bullet E$ and an equivalence of $I^\bullet \Prism_R$-modules
\[ \phi_{E} : \prn*{\Fil^\bullet E \tensor_{\FilN^\bullet \Prism_R} I^\bullet \Prism_R}\pIcpl \to I^\bullet E \mpunc, \]
where in the source we take the $(p, I)$-complete tensor product.
We denote the category of prismatic $F$-gauges over $R$ by $\Fgauge(R)$.
\end{definition}

\begin{remark}
By \cite[Theorem 5.5.10]{fgauges} and \cite[Example 6.1.7]{fgauges}, the category of prismatic gauges, resp.\ $F$-gauges, defined above agrees with the derived category of the stack $R^\Nyg$, resp.\ $R^\Syn$, as defined in \cite{prismatization} and \cite{fgauges}.
\end{remark}

\begin{remark}
\label{Fgaugedescent}
By \cite[Proposition 2.29]{frobheight}, the categories above are sheaves for the quasisyntomic topology on qrsp rings, so we may sensibly extend the definitions by descent to all quasisyntomic rings.
Bhatt and Lurie have announced that the constructions $(-)^\Nyg$ and $(-)^\Syn$ send quasisyntomic covers of quasicompact quasisyntomic formal schemes to fpqc covers, so this definition via descent agrees with the stacky definition \cite[Remark 5.5.18]{fgauges}.%
\footnote{As explained to the author by Bhargav Bhatt, this can be deduced from the analogous statements for the prismatization and the Hodge stack, the former of which is \cite[Lemma 6.3]{blprismatization}.}
\end{remark}

We now introduce some variants of these definitions involving Nygaard completion.

\begin{definition}
For $R$ qrsp, we define the category of \emph{complete prismatic gauges} over $R$ to be the $(p, I)$-complete subcategory
\[ \gaugecpl(R) \defeq \DFcpl(\FilN^\bullet \Prismhat_R)\pIcpl \]
of the complete filtered derived category of the filtered ring $\FilN^\bullet \Prismhat_R$.
\end{definition}

\begin{remark}\label{gaugecplRees}
The category of complete prismatic gauges over a qrsp $R$ may be equivalently defined as the $(t, p, I)$-complete graded derived category of the Rees algebra
\[ \bigoplus_i \FilN^i \Prismhat_R t^{-i} \mpunc. \]
Thus, given a $t$-complete graded module over
\[ \bigoplus_i \FilN^i \Prismhat_R \{i\} t^{-i} \]
(arising, e.g., from $\TCm$) we obtain a prismatic gauge by undoing the Breuil--Kisin twists as in \cite[Remark 5.5.16]{fgauges}.
\end{remark}

\begin{definition}
For $R$ qrsp, a \emph{decompleted complete prismatic gauge} over $R$ consists of a complete prismatic gauge $\Fil^\bullet \widehat{E}$ together with a $(p, I)$-complete $\Prism_R$-module $E$ and a map of $(p, I)$-complete $\Prism_R$-modules
\[ E \to \widehat{E} \mpunc. \]
We let $\gaugedcpl(R)$ denote the category of decompleted complete prismatic gauges over $R$; precisely, it is given by the lax pullback of $\calD(\Prism_R)\pIcpl$ and $\gaugecpl(R)$ over $\calD(\Prism_R)\pIcpl$.
\end{definition}

\begin{remark}
Given a decompleted complete prismatic gauge $(\Fil^\bullet \widehat{E}, E)$ over a qrsp ring $R$, we obtain a natural filtration on $E$ by pullback, making it a filtered module over $\FilN^\bullet \Prism_R$.
This defines a ``decompletion'' functor
\[ \fndec : \gaugedcpl(R) \to \gauge(R) \mpunc. \]
This functor has a left adjoint $\fncpl$ sending a gauge to its completion along with the natural map to the completion from the original underlying module.
\end{remark}

The following result is a variant of the well-known relationship between filtered objects and complete filtered objects in a stable category.

\begin{lemma}\label{gaugedcplequiv}
The functors $\fncpl$ and $\fndec$ are equivalences.
\end{lemma} \begin{proof}
It follows from the definitions that the unit of the adjunction is an isomorphism, and that the counit is an isomorphism exactly on those objects $(\Fil^\bullet \widehat{E}, E)$ such that the map
\[ E \to \widehat{E} \]
exhibits $\widehat{E}$ as the completion of $E$ with respect to the pullback filtration.
But this is always true, as this map will always induce an equivalence on graded pieces.
\end{proof}

We now introduce the Frobenius.

\begin{definition}\label{def:cplFgauge}
For $R$ qrsp, a \emph{complete prismatic $F$-gauge} over $R$ consists of a decompleted complete prismatic gauge $(\Fil^\bullet \widehat{E}, E)$ over $R$ along with an equivalence of complete filtered $I^\bullet \Prism_R$-modules
\[ \phi_E : \prn*{\Fil^\bullet \widehat{E} \tensor_{\FilN^\bullet \Prismhat_R} I^\bullet \Prism_R}_{(p,I),\Fil}^\compl \bij I^\bullet E \mpunc, \]
where in the source we take the $(p, I)$-complete, filtration-complete tensor product.
We let $\Fgaugecpl(R)$ denote the category of complete prismatic $F$-gauges over $R$.
\end{definition}

\begin{remark}
If $(\Fil^\bullet E, \phi_E)$ is a prismatic $F$-gauge over a qrsp ring $R$, then the filtered module
\[ \prn*{\Fil^\bullet E \tensor_{\FilN^\bullet \Prism_R} I^\bullet \Prism_R}\pIcpl \]
is by assumption complete, so the natural map
\[ \prn*{\Fil^\bullet E \tensor_{\FilN^\bullet \Prism_R} I^\bullet \Prism_R}\pIcpl \to \prn*{\Fil^\bullet \widehat{E} \tensor_{\FilN^\bullet \Prismhat_R} I^\bullet \Prism_R}_{(p,I),\Fil}^\compl \]
is an equivalence.
We thus obtain a completion functor
\[ \Fgauge(R) \to \Fgaugecpl(R) \]
compatible with that on gauges.
\end{remark}

The following result is now immediate from \Cref{gaugedcplequiv}.

\begin{corollary}\label{fgaugedcplff}
For any qrsp $R$, the completion functor
\[ \Fgauge(R) \to \Fgaugecpl(R) \]
is exact and fully faithful.
\end{corollary}

\begin{remark}
One can go a step further and consider, for a qrsp $R$, the category of pairs $(\widehat{E}, \phi_{\widehat{E}})$ with $\widehat{E}$ a complete prismatic gauge over $R$ and $\phi_{\widehat{E}}$ a map
\[ \phi_{\widehat{E}} : \Fil^\bullet \widehat{E} \to I^\bullet \widehat{E} \]
linear over the filtered Frobenius morphism $\FilN^\bullet \Prismhat_R \to I^\bullet \Prismhat_R$.
It is not hard to see that this category contains $\Fgaugecpl(R)$, and hence $\Fgauge(R)$, as a full subcategory.
One can view this as a categorification of \cite[Proposition 7.4.6]{apc}.

However, since the objects arising from $\THH$ naturally land in $\Fgaugecpl(R)$, this larger category is unnecessary for us.
It would allow us to define the Quillen formal group on $\Zp^\Syn$ without having to construct $\TPtw$, but, as the Quillen formal group of $\TPtw(R; \Zp)$ is the object which agrees with the Drinfeld formal group over $\Spf \Prism_R$, we are interested in constructing it regardless.
\end{remark}

\begin{remark}[The stacky approach]\label{completedsyntomification}
\newcommand{\Rees}{\mathcal{R}}
The constructions of this section are somewhat ad-hoc and perhaps not the best approach in general, although they suffice for our purposes.
A better approach might be to define for qrsp $R$ an analytic stack (in the sense of Clausen--Scholze \cite{analyticstacks}) $R^{\widehat{\Syn}}$ via a pushout diagram (roughly) of the form
\[\xymatrix{
{\begin{array}{@{}c@{}}
  \Spa\prn*{\Rees\prn*{\FilN^\bullet \Prismhat_R}_t^\compl\brk*{\frac{1}{t}}, \Rees\prn*{\FilN^\bullet \Prismhat_R}_t^\compl} / \Gm \\
  \sqcup \\
  \Spa\prn*{\Rees\prn*{\FilN^\bullet \Prism_R}\brk*{\frac{1}{t}}, \Rees\prn*{\FilN^\bullet \Prism_R}} / \Gm
\end{array}} \ar[d] \ar[r]_{\kern40pt{j_\mathrm{dR} \sqcup j_\mathrm{HT}}} & \Spa\prn*{\Rees\prn*{\FilN^\bullet \Prismhat_R}_t^\compl} / \Gm \ar[d] \\
\Spa\prn*{\Rees\prn*{\FilN^\bullet \Prism_R}\brk*{\frac{1}{t}}, \Rees\prn*{\FilN^\bullet \Prism_R}} / \Gm \ar[r] & R^{\widehat{\Syn}}
\mpunc,
}\]
where we define the Rees ring $\Rees(\Fil^\bullet A)$ as
\[ \Rees(\Fil^\bullet A) \defeq \bigoplus_i \Fil^i A \cdot t^{-i} \mpunc. \]
It seems likely to the author that the natural map $R^{\widehat{\Syn}} \to R^\Syn$ will be an equivalence.
Indeed, this should follow from the statement that the natural diagram of analytic stacks
\[\xymatrix{
\Spa\prn*{\Rees\prn*{\FilN^\bullet \Prismhat_R}_t^\compl\brk*{\frac{1}{t}}, \Rees\prn*{\FilN^\bullet \Prismhat_R}_t^\compl} \ar[d] \ar[r] & \Spa\prn*{\Rees\prn*{\FilN^\bullet \Prismhat_R}_t^\compl} \ar[d] \\
\Spa\prn*{\Rees\prn*{\FilN^\bullet \Prism_R}\brk*{\frac{1}{t}}, \Rees\prn*{\FilN^\bullet \Prism_R}} \ar[r] & \Spa\prn*{\bigoplus_i \FilN^i \Prism_R t^{-i}}
}\]
is a pushout, which should in turn follow from the Beauville--Laszlo-like descent result of \cite[Lecture 11, 1:30:21]{analyticstacks}.
\end{remark}

\section{Graded stacks and graded formal groups}\label{sec:gradedstacks}

In this section, we recall some standard facts about formal groups, and we prove relative affineness of the moduli stack of formal groups over $\BGm$.
The contents of this section beyond \Cref{MfgMUL} are somewhat technical and might be skipped over on a first reading, with its contents referred back to as necessary.
In fact, a reader interested only in \Cref{sec:drinfeldquillenfg} will not need anything but \Cref{quillenfggeneral} and \Cref{gradedfgl}.

Recall that the derived category of the stack $\BGm$ identifies symmetric monoidally with the category of graded objects in $\calD(\bbZ)$, as described (over the sphere spectrum) in \cite{tasosa1modgm}.
Motivated by this, we make the following definition.

\begin{definition}
A \emph{graded stack} is a stack $X$ equipped with a map $X \to \BGm$.
For a graded stack $X$, we refer to the pullback
\[ X \times_{\BGm} * \]
as its \emph{total stack}, and we refer to the line bundle classified by $X \to \BGm$ as the \emph{tautological line bundle}.
We say that $X$ is a \emph{graded scheme}, resp.\ \emph{affine graded scheme}, if its total stack is a scheme, resp.\ affine scheme.
We also define the \emph{graded global sections} $\RGammagr(X)$ to be the pushforward of the unit in $\calD(X)$ along the map $X \to \BGm$.
\end{definition}

\begin{example}
If $A$ is a graded ring, we may regard it as a ring over $\BGm$, so that its relative spectrum
\[ \SpecBGm A \]
defines an affine graded scheme.
In this case, its total stack is just $\Spec A$.

Similarly, if $A$ is a graded ring whose underlying ring is equipped with the adic topology for an ideal $I$ generated by finitely many homogeneous elements, we may define the relative formal spectrum
\[ \SpfBGm A \defeq \colim_n \SpecBGm A/I^n \mpunc, \]
whose total stack is $\Spf A$.%
\footnote{Although this presentation uses the ideal $I$, one may check that the stack obtained depends only on the topology.}
\end{example}

\begin{recollection}\label{gradedderivedrings}
The graded global sections $\RGammagr(X)$ of a graded stack $X$ naturally obtain the structure of an $\Einfty$-ring in $\calD(\BGm)$.
In fact, one can do better than this: descent from the affine case shows that $\RGammagr(X)$ naturally lifts to an object of Mathew's category of \emph{graded derived (commutative) rings}, as developed in \cite[\S 4]{arponhh}.

We denote this category by $\DAlg_{\BGm}$, and given a graded derived ring $A$ we write $\DAlg_A$ for the under category $\prn*{\DAlg_{\BGm}}_{A/}$.
The functor
\[ \RGammagr : \Stk_{\BGm} \to \DAlg_{\BGm}^\opp \]
admits a right adjoint $\SpecBGm$, given explicitly by
\[ \SpecBGm A = \Map_{\DAlg_{\BGm}}(A, -) \mpunc. \]
Note that this restricts to the usual relative spectrum for ordinary graded rings.
\end{recollection}

\begin{definition}
We say that a graded stack is \emph{affine} if the unit map
\[ X \to \SpecBGm \RGammagr(X) \]
is an equivalence.
\end{definition}

\begin{remark}
The notion of an affine stack was first studied by Toën, in the setting of cosimplicial commutative rings \cite{champsaffines}.
\end{remark}

\begin{remark}
In order for this notion of affineness to be well-behaved, we would like the functor $\SpecBGm$ to be fully faithful on the full subcategory of coconnective graded derived rings (i.e.\ those with coconnective underlying complex).
One can reduce this to the analogous claim in the ungraded setting, which is a special case of \cite[Theorem 3.4]{mmaffstk}.
\end{remark}

\begin{construction}
If $R$ is a coconnective graded derived ring, we define a t-structure on the category of graded modules $\Modgr{R}$ by taking the connective objects to be generated under extensions and colimits by the grading-shifts $R(n)$ for $n \in \bbZ$.
This determines an accessible t-structure by \cite[Proposition 1.4.4.11]{ha}, whose coconnective subcategory consists of those modules whose underlying complex is coconnective.
Note that this t-structure is compatible with filtered colimits, i.e.\ coconnectivity is preserved under filtered colimits, and that it is compatible with the symmetric monoidal structure, in the sense of \cite[Example 2.2.1.3]{ha}.
\end{construction}

\begin{remark}
If $X$ is a graded stack, the cohomology functor
\[ \calD(X) \to \calD(\BGm) \]
factors through a lax symmetric monoidal functor $\gamma_*$ landing in the category of graded modules over $\RGammagr(X)$:
\[ \calD(X) \xrightarrow{\gamma_*} \Modgr{\RGammagr(X)} \to \calD(\BGm) \mpunc. \]
This functor has a symmetric monoidal left adjoint
\[ \gamma^* : \Modgr{\RGammagr(X)} \to \calD(X) \]
characterized by the property that, for any graded ring $A$ and map
\[ f : \SpecBGm A \to X \mpunc, \]
the composite $f^* \gamma^*$ identifies functorially with the tensor product $(-) \tensor_{\RGammagr(X)} A$.

Note that the functors $\gamma_*$ and $\gamma^*$ restrict to a symmetric monoidal equivalence between the thick subcategory of $\calD(X)$ generated by powers of the tautological bundle and the thick subcategory of $\Modgr{\RGammagr(X)}$ generated by grading-shifts of the unit.
Note also that, by the argument of \cite[Proposition 7.2.4.2]{ha}, $\Modgr{\RGammagr(X)}$ is compactly generated, and this thick subcategory is exactly the subcategory of compact objects.
\end{remark}

\begin{example}
A \emph{formal group} $\hatG$ on a stack $X$ is an abelian group object (in the sense of \cite[Definition 1.2.4]{ellcoh1}) in the category of stacks over $X$, such that the underlying pointed stack of $\hatG$ is locally on $X$ isomorphic to $(\hat\bbA^1, 0)$.

The moduli stack of formal groups $\Mfg$ admits a natural map to $\BGm$ classifying the cotangent space of the universal formal group.
The total stack of $\Mfg$ is the stack of formal groups equipped with trivializations of their cotangent spaces (often called first-order coordinates).
\end{example}

\begin{remark}
If $X$ is a graded stack, the mapping space (over $\BGm$)
\[ \Mfg(X) \]
identifies with the space of \emph{graded formal groups} over $X$, i.e.\ of pairs consisting of a formal group $\hatG$ over $X$ and an identification of the cotangent space of $\hatG$ with the tautological line bundle.
\end{remark}

\begin{example}\label{quillenfggeneral}
If $E$ is a complex-orientable homotopy-commutative ring spectrum with $\pi_{2*} E$ a commutative ring, its \emph{Quillen formal group} is the graded formal group
\[ \SpfBGm \pi_{2*} \prn*{E^\CPinfty} \]
over $\SpecBGm \pi_{2*} E$, with underlying graded ring the even $E$-cohomology of $\CPinfty$.
Here the group structure is induced from that on $\CPinfty$, and the topology from its CW filtration.
\end{example}

\begin{definition}\label{gradedfgl}
For a graded ring $A$, a \emph{graded formal group law} over $A$ is a formal group law
\[ f(x, y) = \sum_{i, j \geq 0} a_{i j} x^i y^j \]
over $A$ such that each $a_{i j}$ is homogeneous of degree $i + j - 1$.
An isomorphism\footnotemark{} of graded formal group laws $f$ and $g$ is a composition-invertible power series
\footnotetext{%
  Note that an isomorphism of graded formal group laws restricts to a \emph{strict} isomorphism of underlying ungraded formal group laws.
  In the graded setting, it is more natural to always require strictness, in order to ensure compatibility with the maps to $\BGm$.
}
\[ p(x) = \sum_{i \geq 1} b_i x^i \]
such that $b_1 = 1$, each $b_i$ is homogeneous of degree $i - 1$, and
\[ f(p(x), p(y)) = p(g(x, y)) \mpunc. \]
\end{definition}

\begin{example}\label{MfgMUL}
A celebrated theorem of Quillen \cite{quillenmu} identifies the even homotopy ring $\Laz \defeq \pi_{2*} \MU$ of complex cobordism with the Lazard ring, which classifies graded formal group laws.
The associated Quillen formal group over $\SpecBGm \Laz$ is the underlying graded formal group of the universal graded formal group law, and its classifying map
\[ \SpecBGm \Laz \to \Mfg \]
is a fpqc covering.

Further, the ring $\Laz^{(n)} \defeq \pi_{2*} \prn*{\MU^{\tensor n}}$ classifies sequences of graded formal group laws $f_1, \dots, f_n$ together with an isomorphism between $f_i$ and $f_{i+1}$ for each $i$.
These rings together define a cosimplicial graded ring, and the diagram
\[ \SpecBGm \Laz^{(\bullet)} \to \Mfg \]
is a Čech nerve (and thus a colimit diagram).
\end{example}

\begin{definition}
Let $R$ be a graded ring and $\hatG$ a graded formal group over $R$.
Then a \emph{coordinate} on $\hatG$ is an isomorphism of pointed graded $R$-stacks
\[ \hatG \equiv \SpfBGm R\Brk{t} \]
with $t$ in weight $-1$, such that the induced isomorphism of cotangent spaces at the origin agrees with the given isomorphism between the cotangent space of $\hatG$ and the tautological line bundle.
Note that choosing a coordinate on $\hatG$ is equivalent to choosing a lift of $\hatG$ to a graded formal group law.

An \emph{$n$\textsuperscript{th}-order coordinate} on $\hatG$ is an pointed isomorphism between the $n$\textsuperscript{th} order neighborhood of the identity in $\hatG$ and $\SpecBGm R\Brk{t}/(t^n)$, satisfying the same condition on the cotangent space.
\end{definition}

We now turn towards the main goal of this section, which is to show that $\Mfg$ is an affine graded stack.
In order to do this, we will first study the adjunction between $\calD(\Mfg)$ and $\Modgr{\RGammagr(\Mfg)}$.
For this we will need the following technical result, which essentially states that the relative structure sheaf of $\SpecBGm \Laz$ over $\Mfg$ is built from very well-behaved pieces in a very well-behaved way.
A version of this result appears as \cite[Lemma 3.19]{goerssQCohMfg}.\footnotemark
\footnotetext{The author thanks the anonymous referee for pointing this out.}

\begin{proposition}\label{MfgLfiltration}
Let $\underline{\Laz}$ denote the $\calO_{\Mfg}$-algebra arising as the pushforward of the structure sheaf of $\SpecBGm \Laz$.
There is a a $\bbZ_{\geq 1}$-indexed sequence of objects $\underline{\Laz}_i \in \calD(\Mfg)$ such that:
\begin{enumerate}
\item there is an identification
\[ \colim_i \underline{\Laz}_i \equiv \underline{\Laz} \mpunc; \]
\item the object $\underline{\Laz}_1$ identifies with the structure sheaf $\calO_{\Mfg}$; and
\item for each $i \geq 1$, there is a cofiber sequence
\[ \underline{\Laz}_i \to \underline{\Laz}_{i+1} \to \overline{\Laz}_i \mpunc, \]
where $\overline{\Laz}_i$ is a filtered colimit of objects each of which is an iterated extension of powers of the tautological line bundle.
\end{enumerate} \end{proposition}

We note that the sheaves $\underline{\Laz}_i$ above are the Adams--Novikov associated gradeds of Ravenel's spectra $X(i)$ \cite{ravenelXn}.

\begin{proof}
For each $i$ we may define an $\Mfg$-stack $X_i$ classifying $i$\textsuperscript{th}-order coordinates on the universal graded formal group, so that
\[ \SpecBGm \Laz = \lim_i X_i \mpunc. \]
We also write $A$ for the coordinate ring over $\Mfg$ of the universal formal group, which comes with a decreasing filtration $\Fil^\bullet A$ by order.
An $i$\textsuperscript{th}-order coordinate is then a splitting of the map
\[ \Fil^1 A / \Fil^{i+1} A \to \calO(1) \mpunc, \]
so $X_i$ is a torsor for the vector bundle
\[ V_i \defeq \ker\prn*{\Fil^1 A / \Fil^{i+1} A \to \calO(1)} \mpunc. \]
In particular, it is affine, and we define $\underline{\Laz}_i$ to be its coordinate ring.

It remains to check that the $\underline{\Laz}_i$ have the claimed properties.
For (1), note that $\Spec_{\Mfg} \underline{\Laz}$ classifies coordinates on the universal graded formal group, so we have an $\Mfg$-isomorphism
\[ \Spec_{\Mfg} \underline{\Laz} \isom \lim_i X_i \mpunc, \]
and thus, as each of these stacks is a relative affine scheme, an isomorphism
\[ \underline{\Laz} \isom \colim_i \underline{\Laz}_i \mpunc. \]
Property (2) is clear from the definitions, so it remains to check property (3).
For this, note first that as each $X_i$ is a torsor for $V_i$ functorially in $i$, the sheaf $\underline{\Laz}_i$ admits a natural increasing filtration by degree whose associated graded is the symmetric algebra $\Sym V_i^\dual$.
Thus it will suffice to see that, for each $i$ and $k$, the sheaf
\[ \cofib\prn*{\Sym^k V_i^\dual \to \Sym^k V_{i+1}^\dual} \]
admits a finite filtration whose associated graded is a finite sum of powers of the tautological bundle.
Note that the $V_i$ come equipped with functorial finite filtrations arising from the order filtration on $A$, for which the associated graded of the map $V_{i+1} \to V_i$ is the projection
\[ \bigoplus_{n=2}^{i+1} \calO(n) \surj \bigoplus_{n=2}^{i} \calO(n) \mpunc. \]
Thus the map
\[ \Sym^k V_i^\dual \to \Sym^k V_{i+1}^\dual \]
admits a finite filtration whose associated graded is the inclusion
\[ \Sym^k \bigoplus_{n=1}^i \calO(-n) \inj \Sym^k \bigoplus_{n=1}^{i+1} \calO(-n) \mpunc, \]
so the induced filtration on the cofiber has the desired property.
\end{proof}

\begin{proposition}\label{DMfgprops}\begin{enumerate}
\item
  The pullback functors
  \[ \calD(\Mfg) \to \calD(\SpecBGm \Laz) \to \calD(\Laz) \]
  are t-exact.
  Further, an object of $\calD(\Mfg)$ is connective, resp.\ coconnective, if and only if it is connective, resp.\ coconnective, after pullback to $\calD(\Laz)$.
\item
  The $t$-structure on $\calD(\Mfg)$ is left complete.
\item
  The natural map
  \[ \QCoh(\Mfg) \to \calD(\Mfg) \]
  induces an equivalence
  \[ \QCoh(\Mfg) \equiv \calD(\Mfg)^\heart \mpunc. \]
\item
  The natural map
  \[ \calD(\QCoh(\Mfg)) \to \calD(\Mfg) \]
  restricts to an equivalence of coconnective subcategories, and thus exhibits the target as the left completion of the source.
\end{enumerate}\end{proposition} \begin{proof}
Part (1) follows from \cite[\S I.3, Proposition 1.5.4]{gr1}; part (2) from part (1), \Cref{MfgMUL}, and \cite[\S I.3, Lemma 1.5.8]{gr1}; and part (4) from part (3) and \cite[\S I.3, Proposition 2.4.3]{gr1}.%
\footnote{Note that, although the cited results are stated only for Artin stacks in characteristic $0$, the proofs given work just as well in the cases of interest.}
Part (3) follows from part (1) and descent.
\end{proof}

\begin{proposition}\label{MfgQCohaffine}\begin{enumerate}
\item
  The functor
  \[ \gamma_* : \calD(\Mfg) \to \Modgr{\RGammagr(\Mfg)} \]
  and its left adjoint $\gamma^*$ restrict to an equivalence of coconnective subcategories.
  As such, $\gamma^*$ identifies $\calD(\Mfg)$ with the left completion of $\Modgr{\RGammagr(\Mfg)}$.
\item
  The object $\Laz \in \Modgr{\RGammagr(\Mfg)}$ is flat, i.e.\ the tensor product with $\Laz$ is t-exact.
\end{enumerate}\end{proposition}
Part (1) of the above result seems to be closely related to affineness in general; the equivalent result for (ungraded) pointed connected affine stacks over a field has been shown by Lurie in characteristic $0$ \cite{dag8} and by Mondal and Reinecke in arbitrary characteristic \cite{unipotent}.
\begin{proof}
We first prove full faithfulness in (1).

Let $T$ denote the thick subcategory of $\calD(\Mfg)$ generated by powers of the tautological bundle, so that $\gamma_*$ induces an equivalence
\[ \Ind(T) \isom \Modgr{\RGammagr(\Mfg)} \mpunc. \]
Note that by \Cref{DMfgprops}, \cite[\S 3.4]{naumannmfg}, and \cite[Corollary 6.7]{hoveystable}, $T$ identifies with the thick subcategory of $\calD(\Comod_{(\Laz, \Laz^{(2)})})$ generated by the dualizable objects.
Thus $\Ind(T)$ identifies with the stable module category $\Stable_{(\Laz, \Laz^{(2)})}$ of \cite[Definition 4.9]{bhv}.
We may factor $\gamma_*$ as
\begin{align*}
\calD(\Mfg)
  &\to \calD(\QCoh(\Mfg))
\\&\equiv \calD(\Comod_{(\Laz, \Laz^{(2)})})
\\&\to \Stable_{(\Laz, \Laz^{(2)})}
\\&\equiv \Modgr{\RGammagr(\Mfg)} \mpunc,
\end{align*}
so that the first arrow is an equivalence on coconnective objects by \Cref{DMfgprops} and the second is fully faithful, as shown implicitly in \cite[\S 4]{bhv} and explicitly in \cite[\S 3.2]{synthetic}.
As $\gamma_*$ preserves coconnective objects, we find that it restricts to a fully faithful functor
\[ \gamma_*|_{\tau_{\leq 0}} : \tau_{\leq 0} \calD(\Mfg) \to \tau_{\leq 0} \Modgr{\RGammagr(\Mfg)} \mpunc. \]

We now prove (2), which we will use to prove the remainder of (1).
In order to accomplish this, we first note that $\gamma_*|_{\tau_{\leq 0}}$ commutes with filtered colimits.
Indeed, this follows from the above discussion along with \cite[Corollary 3.8]{synthetic}, \cite[Proposition 2.16]{synthetic}, and compatibility of the t-structure on $\Modgr{\RGammagr(\Mfg)}$ with filtered colimits.%
\footnote{Note that our category $\Stable_{(\Laz, \Laz^{(2)})}$ corresponds to the full subcategory of Pstrągowski's on the objects concentrated in even weights.}

We first show that the tensor product with $\Laz$ preserves coconnective objects, i.e.\ that if the underlying complex of $M \in \Modgr{\RGammagr(\Mfg)}$ is coconnective then so is the underlying complex of $M \tensor_{\RGammagr(\Mfg)} \Laz$.
Write $\Laz_i \defeq \gamma_* \underline{\Laz}_i$, where $\underline{\Laz}_i$ is defined as in \Cref{MfgLfiltration}.
As $\gamma_*|_{\tau_{\leq 0}}$, the forgetful functor to complexes, and the functor of homology groups of a complex commute with filtered colimits, it suffices to see that $M \tensor_{\RGammagr(\Mfg)} \Laz_i$ is coconnective for each $i$.
This is clear for $i = 1$, and the mentioned compatibility with colimits implies that $M \tensor_{\RGammagr(\Mfg)} \gamma_* \overline{\Laz}_i$ is coconnective for each $i$, so the statement follows by induction.

To conclude the proof of statement (2) we need to see that the tensor product preserves connective objects.
This follows from the definition of the t-structure along with \Cref{MfgLfiltration} and the fact that $\gamma_*|_{\tau_{\leq 0}}$ commutes with filtered colimits.

Finally, to conclude the proof of statement (1) it remains to check that the left adjoint
\[ \tau_{\leq 0} \gamma^*|_{\tau_{\leq 0}} : \tau_{\leq 0} \Modgr{\RGammagr(\Mfg)} \to \tau_{\leq 0} \calD(\Mfg) \]
of the functor above is conservative, and that $\gamma^*$ preserves coconnectivity.

For preservation of coconnectivity, note that (2) implies that the composite
\[ \Modgr{\RGammagr(\Mfg)} \to \calD(\Mfg) \to \calD(\Laz) \mpunc, \]
which is given by the tensor product, preserves coconnective objects (as coconnectivity in the source and target are both detected on underlying complexes), and thus by \Cref{DMfgprops} that $\gamma^*$ does as well.

Finally, for conservativity on coconnective objects it suffices to show that if $M \in \tau_{\leq 0} \Modgr{\RGammagr(\Mfg)}$ is nonzero then so is $M \tensor_{\RGammagr(\Mfg)} \Laz$.
Writing
\[ U : \Modgr{\RGammagr(\Mfg)} \to \calD(\BGm) \]
for the forgetful functor, we may assume without loss of generality that $\pi_0 U(M)$ is nonzero.
It will thus suffice to see that the map
\[ \pi_0 U(M) \to \pi_0 U \prn*{M \tensor_{\RGammagr(\Mfg)} \Laz} \]
is injective.
By compatibility with filtered colimits and induction on $i$, it is enough to show that for each $i$ the map
\[ \pi_0 U \prn*{M \tensor_{\RGammagr(\Mfg)} \Laz_i} \to \pi_0 U \prn*{M \tensor_{\RGammagr(\Mfg)} \Laz_{i+1}} \]
is injective, which is immediate from the fact that $M \tensor_{\RGammagr(\Mfg)} \epsilon_* \overline{\Laz}_i$ is coconnective.
\end{proof}

\begin{remark}\label{jacobaffineness}
One may give a somewhat simpler proof of the above proposition (with essentially the same underlying idea) as a consequence of the unipotence of the group scheme parametrizing automorphisms of the $1$-dimensional formal disk which are trivial up to first order.\footnotemark
\footnotetext{The author learned this from Jacob Lurie; in the time since this work first appeared, a version of this proof has appeared as \cite[Example 4.6]{mmaffstk}.}
\end{remark}

\begin{corollary}\label{Mfgclassicalheart}
If $M \in \Modgr{\RGammagr(\Mfg)}$ has underlying graded complex concentrated in degree $0$, then $M$ lies in the heart $(\Modgr{\RGammagr(\Mfg)})^\heart$.
\end{corollary} \begin{proof}
It is clear that $M$ is coconnective, so we just need to check that it is connective.
Explicitly, we need to see that for any $(-1)$-coconnective $N$, the group $\Hom_{\Modgr{\RGammagr(\Mfg)}}(M, N)$ is trivial.
By \Cref{MfgQCohaffine} we may equivalently show triviality of $\Hom_{\Mfg}(\gamma^* M, \gamma^* N)$.
It thus suffices by \Cref{DMfgprops} to show connectivity of $M \tensor_{\RGammagr(\Mfg)} \Laz$ and $(-1)$-coconnectivity of $N \tensor_{\RGammagr(\Mfg)} \Laz$ as $\Laz$-modules.
The former follows from \Cref{MfgLfiltration} as in the proof of \Cref{MfgQCohaffine}, while the latter is an immediate consequence of \Cref{MfgQCohaffine}(2).
\end{proof}

\begin{proposition}\label{Mfgaffine}
The graded stack $\Mfg$ is affine.
\end{proposition} \begin{proof}
We have a morphism of augmented simplicial objects
\[\xymatrix{
\SpecBGm \Laz^{(\bullet)} \ar[d] \ar@{=}[r] & \SpecBGm \Laz^{(\bullet)} \ar[d] \\
\Mfg \ar[r] & \SpecBGm \RGammagr(\Mfg) \mpunc,
}\]
which we wish to show is an equivalence.
Here the left vertical arrow is a Čech nerve.
From this and \Cref{MfgQCohaffine} we see that the diagram
\[ \RGammagr(\Mfg) \to \Laz^{(\bullet)} \]
is a Čech conerve in $(\Modgr{\RGammagr(\Mfg)})^\heart$, hence by flatness of $\Laz$ a Čech conerve in $\Modgr{\RGammagr(\Mfg)}$.
Thus the augmented simplicial objects above are both Čech nerves.
It will therefore suffice to see that the map
\[ \SpecBGm \Laz \to \SpecBGm \RGammagr(\Mfg) \]
is an fpqc surjection.
For this it is enough to show that, for any ordinary ring $R$ with a map
\[ \Spec R \to \SpecBGm \RGammagr(\Mfg) \mpunc, \]
the pullback
\[ \Spec R \times_{\SpecBGm \RGammagr(\Mfg)} \SpecBGm \Laz \to \Spec R \]
is faithfully flat.
Localizing, we may assume that the induced map $\Spec R \to \BGm$ classifies the trivial line bundle; unpacking everything, we need to see that, given a map of graded derived rings
\[ \RGammagr(\Mfg) \to R[t^{\pm 1}] \]
with $R$ classical and $\abs{t} = 1$, the induced map
\[ R \to \gr^0\prn*{R[t^{\pm 1}] \tensor_{\RGammagr(\Mfg)} \Laz} \]
is a faithfully flat map of classical rings.
\Cref{Mfgclassicalheart} shows that $R[t^{\pm 1}]$ is in the heart of $\Modgr{\RGammagr(\Mfg)}$, and by flatness of $\Laz$ so is the tensor product.
The desired statement therefore follows from \Cref{MfgQCohaffine} and faithful flatness of
\[ \SpecBGm \Laz \to \Mfg \mpunc. \qedhere \]
\end{proof}

\begin{remark}
One can also give a proof of \Cref{Mfgaffine} using \Cref{MfgQCohaffine} and the results of \cite{stefanichtannaka}.
\end{remark}

\begin{remark}\label{affdiagaffineness}
In \Cref{Mfgaffine}, the specific choice of derived rings as our version of coconnective commutative rings was inessential, and the proof goes through just as well if we instead use, e.g., $\Einfty$-rings.
In fact, a careful analysis of the proof shows that for any relative geometric stack with affine diagonal, all notions of relative affineness (corresponding to reasonable notions of coconnective commutative ring) is equivalent.
This is related to the fact that Tannaka duality holds for such stacks, so they can be recovered from their symmetric monoidal derived categories, without any extra structure pertaining to enhancements of the $\Einfty$ structure.
\end{remark}

\begin{remark}
As a consequence of \Cref{Mfgaffine}, if $X$ is a graded stack and $\hatG$ a graded formal group on $X$, we obtain a natural descent of $\hatG$ to $\SpecBGm \RGammagr(X)$, which we will denote by $\hatG^\aff$; this construction provides an equivalence between the spaces of graded formal groups on $X$ and on $\SpecBGm \RGammagr(X)$, whose inverse is given by pullback.
Explicitly, we can write $\hatG^\aff$ as
\[ \colim_n \RGammagr(\hatG_n) \mpunc, \]
where $\hatG_n$ is the $n$\textsuperscript{th}-order neighborhood of the identity in $\hatG$.
\end{remark}

The graded-affineness of $\Mfg$ allows us to construct ``pseudo-algebraizations'' of certain formal groups, as follows.

\begin{lemma}\label{bddfgtensorcomplete}
Let $X$ be a graded stack such that $\RGammagr(X)$ is concentrated in finitely many weights, and let $\hatG$ be a graded formal group on $X$.
Then for each $n \geq 0$, the natural map
\[ \RGammagr(\hatG)^{\tensor_{\RGammagr(X)} n} \to \RGammagr(\hatG^n) \]
is an equivalence.
\end{lemma} \begin{proof}
Writing $\Guniv$ for the universal formal group, we may regard $\RGammagr(\Guniv)$ as an abelian cogroup in complete filtered graded $\Einfty$-$\RGammagr(\Mfg)$-algebras,%
\footnote{We switch to the $\Einfty$ setting here in order to avoid having to introduce filtered graded derived rings, and because it makes no difference for the result.}
where the filtration is the decreasing, $\bbN$-indexed adic filtration.
By base change from the universal case, $\RGammagr(\hatG)$ is such an object over $\RGammagr(X)$, and the map in question is given by completion with respect to the filtration.
It will thus suffice to show that the filtration on the source is already complete.

Note that the $k$\textsuperscript{th} graded piece of the filtration on $\RGammagr(\hatG)$ is given by the shift $\RGammagr(X)(-k)$, so the $k$\textsuperscript{th} graded piece of the filtration on $\RGammagr(\hatG)^{\tensor_{\RGammagr(X)} n}$ is given by the direct sum
\[ \bigoplus_{\substack{i_1 + \dots + i_n = k \\ i_1, \dots, i_n \geq 0}} \RGammagr(X)(-k) \mpunc. \]
Thus, as $\RGammagr(X)$ lives in bounded weights, the induced filtration on each piece of the internal grading is finite, so the filtration is complete.
\end{proof}

\begin{construction}\label{bddfgalmostalg}
In the situation of \Cref{bddfgtensorcomplete}, the abelian group structure on $\hatG$ gives $\RGammagr(\hatG)$ the structure of an abelian cogroup in $\DAlg_{\RGammagr(X)}$.
We thus obtain a natural abelian group stack $G^\aff \defeq \SpecBGm \RGammagr(\hatG)$ over $\SpecBGm \RGammagr(X)$, along with a natural map
\[ \hatG^\aff \to G^\aff \mpunc. \]
\end{construction}

It is not clear whether this construction is well-behaved in complete generality, though we will see that it is in our primary case of interest.

\begin{question}
In the situation of \Cref{bddfgalmostalg}, does the group stack $G^\aff$ have underlying pointed stack locally isomorphic to $(\bbA^1, 0)$?
Does the map
\[ \hatG^\aff \to G^\aff \]
exhibit the source as the completion of the target at the identity section?
\end{question}

Note that the answer is yes, for instance, whenever there is an ordinary graded ring $R$ concentrated in finitely many weights such that $\SpecBGm \RGammagr(X)$ admits a flat cover by $\SpecBGm R$.

\section{Identification of the Drinfeld and Quillen formal groups}\label{sec:drinfeldquillenfg}

In this section, we will prove the first of our main theorems.
Our strategy is to use the $F$-crystal structures on the objects of interest to reduce to checking that two $(\phi, \Gal_{\Qp})$-structures on the multiplicative formal group over $W(\Cpflat)$ agree.
Torsionfreeness of $W(\Cpflat)$ reduces this to agreement of the two induced structures on the cotangent space, which we may check explicitly.

We begin by introducing the objects of interest.

\begin{construction}[The Quillen formal group on $\Zp^\Syn$]
For a qrsp ring $R$ the object
\[ \bigoplus_i \pi_{2i} \prn*{\TCm(R; \Zp)^\CPinfty} \{-i\} t^{-i} \]
defines a Hopf algebra in $\Pro(\gaugecpl(R))$ as in \Cref{gaugecplRees}, with pro- structure coming from the CW filtration on $\CPinfty$.
The composite
\[ \pi_0 \prn*{\TPtw(R; \Zp)^\CPinfty} \to \pi_0 \prn*{\TP(R; \Zp)^\CPinfty} \jib \colim_i \pi_{2i} \prn*{\TCm(R; \Zp)^\CPinfty}\{-i\} \]
lifts this Hopf algebra to $\Pro(\gaugedcpl(R))$.
We also have a natural Frobenius-linear map
\[ \pi_{2*} \prn*{\TCm(R; \Zp)^\CPinfty} \to \pi_{2*} \prn*{\TPtw(R; \Zp)^\CPinfty} \mpunc, \]
which yields (after Breuil--Kisin twisting) a filtered Frobenius-linear map from our pro-completed-gauge to
\[ I^\bullet \pi_0 \prn*{\TPtw(R; \Zp)^\CPinfty} \mpunc. \]
It is immediate from the construction of $\TPtw$ that this map is an isomorphism after linearization as in \Cref{def:cplFgauge}, so we obtain a lift of our Hopf algebra to $\Pro(\Fgaugecpl(R))$.

As each term of this pro-system is an iterated extension of Breuil--Kisin twists of the unit, we obtain from \Cref{fgaugedcplff} a Hopf algebra in $\Pro(\Fgaugevect(R))$, where $\Fgaugevect(R)$ is the category of vector bundles on the stack $R^\Syn$.
Now by naturality this descends to a Hopf algebra in $\Pro(\Fgaugevect(\Zp))$ (see \Cref{Fgaugedescent}), defining a group stack $\GQ$ over $\Zp^\Syn$, which we call the \emph{Quillen formal group}.%
\footnote{Note that $\GQ$ is indeed a formal group: we may by descent check this on an fpqc cover, so it suffices to note that the pullback of $\GQ$ to $\OCp^\Nyg$ is the Quillen formal group of $\TCm(\OCp; \Zp)$ (in the sense of \Cref{quillenfggeneral}).}
\end{construction}

\begin{remark}
Note that if $E$ is a complex-orientable ring spectrum then the cotangent space of its Quillen formal group is the tautological line bundle $\calO(1)$ over $\SpecBGm \pi_{2*} E$.
Similarly, it follows from the construction above that the cotangent space of $\GQ$ is the first Breuil--Kisin twist $\calO_{\Zp^\Syn}\{1\}$.
\end{remark}

\begin{remark}
As $\TC(-; \Zp)$ is quasisyntomic-locally even \cite[Theorem 14.1]{prisms}, one can also define $\GQ$ by using the even homotopy sheaves of $\TC(-; \Zp)^\CPinfty$ to locally define a prismatic $F$-gauge without passing through the Nygaard-completion.
\end{remark}

\begin{recollection}[The Drinfeld formal group {\cite{drinfeldfg}}]
Drinfeld constructs a certain formal group $\GDr$ over $\Zp^\Prism$ equipped with a ``co-Frobenius'' morphism
\[ \phi_{\Zp^\Prism}^* \GDr \to \GDr \mpunc. \]
After pullback to the $q$-de Rham prism, this formal group is given by $\Spf \Zp\Brk{q-1, z}$, with formal group law
\[ z \mapsto z_1 + z_2 + (q-1) z_1 z_2 \mpunc. \]
This pullback obtains a natural $\Zptimes$-action from that on the $q$-de Rham prism, explicitly determined by the formula
\[ \Zptimes \ni a : (q-1) z \mapsto ((q-1) z + 1)^a - 1 \mpunc, \]
and the co-Frobenius map is given by sending
\[ z \mapsto [p]_q z \mpunc. \]
In particular, the co-Frobenius map is an isomorphism after $p$-completely inverting $[p]_q$, so the coordinate ring of $\GDr$ obtains the structure of a Hopf algebra in the category
\[ \Pro(\Fcrysvect((\Zp)_\Prism, \calO_\Prism)) \]
of pro-objects in prismatic $F$-crystals of vector bundles on $(\Zp)_\Prism$ (as defined in \cite[Definition 4.1]{fcrystalsGaloisreps}).
\end{recollection}

We now introduce some categories which will allow us to reduce the study of formal groups over $\Zp^\Syn$ to those over $W(\Cpflat)$.

\begin{recollection}
The category $\Fcrysvect((\Zp)_\Prism, \calO_\Prism)$ admits a forgetful functor
\[ \Fcrysvect((\Zp)_\Prism, \calO_\Prism) \to \Fcrysvect((\Zp)_\Prism, \calO_\Prism[\tfrac{1}{\calI}]\pcpl) \]
to the category of \emph{Laurent $F$-crystals} of vector bundles on $(\Zp)_\Prism$, as defined in \cite[Definition 3.2]{fcrystalsGaloisreps}.
Base changing to $\OCp$, we obtain a functor
\[ \Fcrysvect((\Zp)_\Prism, \calO_\Prism[\tfrac{1}{\calI}]\pcpl) \to (\phi, \Gal_{\Qp})\mathhyphen\mathrm{mod}^\vect(W(\Cpflat)) \]
to the category of Galois-equivariant Laurent $F$-crystals of vector bundles on $(\OCp)_\Prism$, which we identify with the category of $(\phi, \Gal_{\Qp})$-equivariant vector bundles over $W(\Cpflat)$.
\end{recollection}

\begin{proposition}\label{FgaugephiGammamod}
Each of the functors
\begin{align*}
\Fgauge^\vect(\Zp)
  &\to \Fcrysvect((\Zp)_\Prism, \calO_\Prism)
\\&\to \Fcrysvect((\Zp)_\Prism, \calO_\Prism[\tfrac{1}{\calI}]\pcpl)
\\&\to (\phi, \Gal_{\Qp})\mathhyphen\mathrm{mod}^\vect(W(\Cpflat))
\end{align*}
is fully faithful, and the last is an equivalence.
\end{proposition} \begin{proof}
The first functor is fully faithful by \cite[Theorem 6.6.13]{fgauges} or \cite[Corollary 2.53]{frobheight}, the second fully faithful by \cite[Theorem 1.2, Corollary 3.7]{fcrystalsGaloisreps}, and the last an equivalence by the proof of \cite[Theorem 5.6]{wufcrys}.
\end{proof}

\begin{theorem}\label{drinfeldquillenfg}
The Drinfeld and Quillen formal groups define the same Hopf algebra in $\Pro(\Fcrysvect((\Zp)_\Prism, \calO_\Prism))$.
In particular, the two formal groups agree on $\Zp^\Prism$, and we have a natural isomorphism of formal groups
\[ \GDr_{\Prism_R} \isom \Spf \pi_0 \prn*{\TPtw(R; \Zp)^\CPinfty} \]
for any qrsp ring $R$.
\end{theorem} \begin{proof}
By \Cref{FgaugephiGammamod} it will suffice to see that they define the same $(\phi, \Gal_{\Qp})$-equivariant formal group over $W(\Cpflat)$.
After this base change we may write the coordinate ring of the Drinfeld formal group as $W(\Cpflat)\Brk{z}$, with formal group law
\[ z \mapsto z_1 + z_2 + \mu z_1 z_2 \mpunc, \]
Frobenius $z \mapsto \tilde\xi^{-1} z$, and Galois action
\[ \Gal_{\Qp} \ni g : \mu z \mapsto (\mu z + 1)^{\chicyc(g)} - 1 \]
(where $\chicyc$ is the cyclotomic character).
Note that the Galois action on the cotangent space $(z) / (z^2)$ is thus given by
\[ \Gal_{\Qp} \ni g : \mu z \mapsto \chicyc(g) \mu z \mpunc. \]

To identify $\GQ$, note that we have a $\Gal_{\Qp}$-equivariant map
\[ \TC(\OCp; \Zp) \to \TCm(\OCp; \Zp) \mpunc, \]
where the source identifies with $\ku\pcpl$ by the argument of \cite[Corollary 1.3.8]{hntc}.
Recall that $\ku$ admits a complex orientation with associated formal group law
\[ z' \mapsto z'_1 + z'_2 + \beta z'_1 z'_2 \mpunc, \]
where $\beta$ is the Bott element and $z'$ is the complex orientation, regarded as a weight-$(-1)$ element of $\pi_{2*} \prn[\big]{\ku^\CPinfty}$.
We may identify $\beta$ with $\log_\Prism \epsilon \in \pi_2 \TC(\OCp; \Zp)$ and thus describe $\GQ$ over the twisted Rees ring $\bigoplus_i \FilN^i \Prism_{\OCp} \{i\} t^{-i}$ by the graded formal group law
\[ z' \mapsto z'_1 + z'_2 + (\log_\Prism \epsilon) t^{-1} z'_1 z'_2 \mpunc. \]
After $p$-completely inverting $\xi$ we may set $z \defeq \frac{\log_\Prism \epsilon}{\mu} t^{-1} z'$ to obtain the formal group law
\[ z \mapsto z_1 + z_2 + \mu z_1 z_2 \]
over the graded ring
\[ \prn*{\bigoplus_i \FilN^i \Prism_{\OCp} \{i\} t^{-i}}\brk*{\frac{1}{\xi}}\pcpl \isom \bigoplus_i W(\Cpflat)\{i\} \mpunc. \]
Note that the element $z$ and the associated formal group law live in weight $0$, so we may regard $z$ as a coordinate for the Quillen formal group over $W(\Cpflat)$, with associated formal group law as above.

In particular, as $\mu$ is invertible in $W(\Cpflat)$, $\GQ_{W(\Cpflat)}$ is isomorphic to the multiplicative formal group.
Recall that an automorphism of the multiplicative formal group law
\[ x \mapsto x_1 + x_2 + x_1 x_2 \mpunc, \]
over a torsion-free ring $R$ is given by
\[ x \mapsto (x + 1)^a - 1 \]
for some $a \in R^\times$ such that $\binom{a}{k} \in R$ for all $k \geq 1$.\footnotemark
\footnotetext{%
  To see this, note that any such automorphism is determined by its base change to $R[\tfrac{1}{p}]$, where the multiplicative and additive formal groups are isomorphic by the exponential map.
  Thus, up to conjugation by the exponential, our automorphism is given by an additive polynomial, which (as we are in characteristic $0$) is necessarily linear.
  Unravelling this yields the claim.
}
In particular, such an automorphism is determined entirely by $a$.
Thus the Galois and Frobenius actions on $\GQ_{W(\Cpflat)}$ are determined by the induced actions on its cotangent space.
As $z'$ is a complex orientation, its image in the cotangent space is by definition $1$, so the image of $z$ is $\bar{z} \defeq \frac{\log_\Prism \epsilon}{\mu}$.
The Galois action on $\mu \bar{z}$ is therefore given by multiplication by the cyclotomic character, and by \cite[Remark 2.6.2]{apc} the Frobenius map sends $\bar{z} \mapsto \tilde{\xi}^{-1} \bar{z}$, as desired.
\end{proof}

\begin{corollary}\label{drinfeldfgZpSyn}
The Drinfeld formal group extends to a formal group over $\Zp^\Syn$ with cotangent space $\calO_{\Zp^\Syn}\{1\}$, and this extension is uniquely determined by its cotangent space.
More precisely, there is a unique lift
\[\xymatrix{
\Zp^\Prism \ar[d]_{\GDr} \ar[r] & \Zp^\Syn \ar@{-->}[dl] \ar[d]^{\calO_{\Zp^\Syn}\{1\}} \\
\Mfg \ar[r] & \BGm
\mpunc.
}\]
\end{corollary} \begin{proof}
We have just seen that such an extension exists, so it remains only to show uniqueness.
Let $\GDr_{W(\Cpflat)}$ denote the $\Gal_{\Qp}$-equivariant formal group over $W(\Cpflat)$ obtained from $\GDr$ by base change.
We have a canonical identification of the cotangent space of $\GDr_{W(\Cpflat)}$, considered as a $\Gal_{\Qp}$-equivariant $W(\Cpflat)$-module, with $W(\Cpflat)\{1\}$.
By \Cref{FgaugephiGammamod} it will suffice to see that the $\Gal_{\Qp}$-equivariant formal group $\GDr_{W(\Cpflat)}$ admits at most one Frobenius structure such that the induced structure on its cotangent space is equal to the usual Frobenius structure on $W(\Cpflat)\{1\}$.
But this follows as above from the classification of automorphisms of the multiplicative formal group.
\end{proof}

\section{Algebraization of the Drinfeld formal group}\label{sec:algebraization}

We now direct our efforts towards the second of our main theorems.
In light of \Cref{drinfeldquillenfg} and \Cref{drinfeldfgZpSyn}, we ignore the distinction between the Drinfeld and Quillen formal groups, and in particular regard $\GDr$ as a formal group over $\Zp^\Syn$ rather than only over $\Zp^\Prism$.

Our strategy here will be to use the $(p, v_1)$-completeness of $\Zp^\Syn$ along with a Beilinson--Lichtenbaum-type statement to reduce to a situation where we can apply \Cref{bddfgalmostalg}.
It is then a simple matter to check that the pseudo-algebraization we obtain is actually the desired algebraization.

We first introduce the elements $v_1^{p^{r-1}}$.
This is a generalization of \cite[Construction 2.7]{bmsyntomic}, and follows the same strategy.
\begin{construction}[The class $v_1^{p^{r-1}}$]\label{v1construction}
\newcommand{\Zcyc}{\bbZ[\zeta_{p^\infty}]}
For $r \geq 1$, we define as follows a certain class
\[ v_1^{p^{r-1}} \in \rmH^0(\bbZ/p^r(p^{r-1} (p-1))(\bbZ)) \]
in the mod-$p^r$ syntomic cohomology, as defined in \cite[\S 8.4]{apc}, of $\bbZ$.
The canonical system of roots of unity defines an element $\epsilon \in H^0(\bbZ/p^r(1)(\Zcyc))$; we claim that $\epsilon^{p^{r-1} (p-1)}$ descends to $\bbZ$, giving the desired class.
For this, we need to see that it lies in the equalizer of the two maps
\[ \rmH^0(\bbZ/p^r(p^{r-1} (p-1))(\Zcyc)) \rightrightarrows \rmH^0(\bbZ/p^r(p^{r-1} (p-1))(\Zcyc \tensor \Zcyc)) \mpunc. \]
Over $\Zcyc \tensor \Zcyc [\tfrac{1}{p}]$ this follows from the fact that $\epsilon^{p^r (p-1)}$ identifies with the canonical generator of the étale sheaf $\mu_{p^r}^{\tensor p^{r-1} (p-1)}$ over $\Zcyc[\tfrac{1}{p}]$, so it remains only to check over the $p$-completion $\Zpcyc \tensorhat \Zpcyc$.

Concretely, we have two classes $\epsilon_1, \epsilon_2 \in T_p(\Zpcyc \tensorhat \Zpcyc)^\times$, and we need to check that the elements
\[ (\log_\Prism \epsilon_1)^{p^{r-1} (p-1)}, (\log_\Prism \epsilon_2)^{p^{r-1} (p-1)} \in \Prism_{\Zpcyc \tensorhat \Zpcyc} \{p^{r-1} (p-1)\} / p^r \]
coincide.
Write $q_1, q_2 \in \Prism_{\Zpcyc \tensorhat \Zpcyc}$ for the two images of $q \in \Prism_{\Zpcyc} \isom \QdRcyc$.
By $(p, q-1)$-complete flatness of the two maps
\[ \QdR \rightrightarrows \Prism_{\Zpcyc \tensorhat \Zpcyc} \]
and \cite[Lemma 5.15]{cmaic}, the elements
\[ (q_1 - 1), (q_2 - 1) \in \Prism_{\Zpcyc \tensorhat \Zpcyc} / p^r \]
are nonzerodivisors, so we may check the equality in question after base change to
\[ \Prism_{\Zpcyc \tensorhat \Zpcyc} / p^r [\tfrac{1}{(q_1 - 1)(q_2 - 1)}] \mpunc. \]
By \cite[Example 2.7.5]{apc} the elements
\[ \log_\Prism \epsilon_1, \log_\Prism \epsilon_2 \in \Prism_{\Zpcyc \tensorhat \Zpcyc} \{1\} / p^r [\tfrac{1}{(q_1 - 1)(q_2 - 1)}] \]
are generators, so there exists some unit
\[ x \in \Prism_{\Zpcyc \tensorhat \Zpcyc} / p^r [\tfrac{1}{(q_1 - 1)(q_2 - 1)}] \]
such that $\log_\Prism \epsilon_1 = x \log_\Prism \epsilon_2$.
Since these elements are fixed by the divided Frobenius, we have $\phi(x) = x$, so by \Cref{frobliftcongruence} we find that $x^{p^r} = x^{p^{r-1}}$.
Thus $x^{p^{r-1} (p-1)} = 1$, implying the desired equality.
\end{construction}

In the \namecref{v1construction} above we used the following lemma.

\begin{lemma}\label{frobliftcongruence}
Suppose $R$ is a ring with a lift of Frobenius $\phi$, and let $x$ be an element of $R$.
Then
\[ \phi(x)^{p^{r-1}} \cong x^{p^r} \pmod{p^r} \]
for each $r \geq 1$.
\end{lemma} \begin{proof}
This holds by definition for $r = 1$, so assume inductively that
\[ \phi(x)^{p^{r-2}} \cong x^{p^{r-1}} \pmod{p^{r-1}} \mpunc. \]
Then
\[ \phi(x)^{p^{r-2}} = x^{p^{r-1}} + p^{r-1} y \]
for some $y \in R$, so
\begin{align*}
\phi(x)^{p^{r-1}}
  &= \prn*{x^{p^{r-1}} + p^{r-1} y}^p
\\&= x^{p^r} + \sum_{i=1}^p \binom{p}{i} x^{p^{r-1} (p-i)} p^{(r-1) i} y^i
\\&\cong x^{p^r} \pmod{p^{r-1}}
\mpunc.
\qedhere
\end{align*}
\end{proof}

Note that $v_1^{p^r}$ is equal to $\prn*{v_1^{p^{r-1}}}^p$ after reduction modulo $p^r$, justifying the above notation.


\begin{proposition}[The étale comparison, {\cite[Theorem 5.1]{bmsyntomic}}]\label{etalecomparison}
For any qcqs derived scheme $X$, the natural map
\[ \bigoplus_n \bbZ/p^r(n)(X) \to \bigoplus_n \bbZ/p^r(n)(X[\tfrac{1}{p}]) \]
identifies the target with the localization of the source at $v_1^{p^{r-1}}$.
\end{proposition} \begin{proof}
As the image of $v_1^{p^{r-1}}$ in the target is invertible, we just need to show that the induced map
\[ \prn*{\bigoplus_i \bbZ/p^r(n)(X)}\brk*{\frac{1}{v_1^{p^{r-1}}}} \to \bigoplus_i \bbZ/p^r(n)(X[\tfrac{1}{p}]) \]
is an equivalence.
This may be checked modulo $p$, where it is \emph{loc.\ cit.}
\end{proof}

\begin{proposition}[$(p, v_1)$-completeness of $\Zp^\Syn$]\label{v1nilpotent}
The natural map
\[ \colim_r\ (\Zp^\Syn)_{0 = p^r = v_1^{p^{r-1}}} \to \Zp^\Syn \]
is an equivalence.
\end{proposition} \begin{proof}
It suffices to check this after base change by the fpqc cover
\[ \Spf \prn*{\bigoplus_n \FilN^n \Prism_\Zpcyc \{n\} t^{-n}}\pIcpl \to \Zp^\Syn \]
of the target.
Here we may identify $v_1$ with $\prn*{(\log_\Prism \epsilon) t^{-1}}^{p-1}$, so we are reduced to checking that the coordinate ring of the source is $\prn*{p, (\log_\Prism \epsilon) t^{-1}}$-complete.
Note that it is $(p, q-1)$-complete because $(q-1)^{p-1} \cong [p]_q \pmod{p}$.
By \cite[Proposition 2.6.1]{apc} we know that $(\log_\Prism \epsilon) t^{-1}$ is divisible by $q^{1/p} - 1$, so the desired statement follows from the inclusion
\[ (p, q^{1/p} - 1)^p \subseteq (p, \phi(q^{1/p} - 1)) = (p, q-1) \mpunc. \qedhere \]
\end{proof}

In the following we regard the syntomification of a ring as a graded stack via the map to $\BGm$ classifying the Breuil--Kisin line bundle $\calO\{1\}$.
We do the same for the vanishing loci of powers of $p$ and $v_1$.

\begin{proposition}[$p$-adic Beilinson--Lichtenbaum for $\Zp$]\label{quillenlichtenbaum}
For each $r \geq 1$ there exists some finite $n$ such that the graded complex
\[ \RGammagr\prn[\Big]{(\Zp^\Syn)_{0 = p^r = v_1^{p^{r-1}}}} \]
is concentrated in weights $[0, n]$.
\end{proposition} \begin{proof}
By \Cref{etalecomparison} and the vanishing of syntomic cohomology in negative weights \cite[Proposition 7.16]{bms2}, it suffices to see that the syntomic cohomology of $\Zp$ agrees with the étale cohomology of $\Qp$ in sufficiently high weights.
But this is well-known, e.g.\ by \cite[Proposition 5.2]{bmsyntomic} and the fact that both are concentrated in uniformly bounded degrees.
\end{proof}

\begin{theorem}\label{drinfeldfgalg}
The Drinfeld formal group admits an algebraization over $\Zp^\Syn$, in the sense of \cite[\S 2.12.1]{drinfeldfg}.
After pullback to $\Zp^\Prism$, this is the algebraization of \cite[Conjecture 2.12.4]{drinfeldfg}.
\end{theorem} \begin{proof}
Recall that the cotangent space of $\GDr$ is the Breuil--Kisin line bundle $\calO_{\Zp^\Syn}\{1\}$.
Thus from \Cref{quillenlichtenbaum} and \Cref{bddfgalmostalg} we obtain, naturally in each $r \geq 1$, a map of group stacks
\[ \GDr_{0 = p^r = v_1^{p^{r-1}}} \to \GDralg_{0 = p^r = v_1^{p^{r-1}}} \]
over the stack $(\Zp^\Syn)_{0 = p^r = v_1^{p^{r-1}}}$.
By \Cref{v1nilpotent} we obtain a map of group stacks
\[ \GDr \to \GDralg \]
over $\Zp^\Syn$ in the colimit.
We first need to check that this map is actually an algebraization, i.e.\ that its target is a smooth affine group scheme and that the map is the inclusion of the completion at the identity.
As these properties may be checked locally, it suffices to show this after base change to $\OCp^\Syn$.
Setting $\beta \defeq \log_\Prism \epsilon$, we see by \cite[Theorem 1.3.6]{hntc} that
\[ \RGammagr(\OCp^\Syn) \isom \Zp[\beta] \mpunc, \]
with $\beta$ in weight $1$.
Note that
\[ \beta^{p^{r-1} (p-1)} \cong v_1^{p^{r-1}} \pmod{p^r} \]
for all $r \geq 1$, so we may apply \Cref{bddfgalmostalg} again to obtain a map of group stacks
\[ \GDraff_{\OCp^\Syn} \to \GDraffalg_{\OCp^\Syn} \]
over the stack
\[ \colim_r \SpecBGm \RGammagr\prn[\Big]{(\OCp^\Syn)_{0 = p^r = v_1^{p^{r-1}}}} \isom \SpfBGm \Zp\Brk{\beta} \mpunc. \]
It will suffice to check the properties of interest here.
Using \Cref{drinfeldquillenfg} and the identification $\ku\pcpl \equiv \TC(\OCp; \Zp)$ described in \cite[Corollary 1.3.8]{hntc} we find
\[ \GDraff_{\OCp^\Syn} \isom \SpfBGm \Zp\Brk{\beta, z} \]
with formal group law
\[ z \mapsto z_1 + z_2 + \beta z_1 z_2 \mpunc. \]
Modulo $(p, \beta)^n$, \Cref{bddfgalmostalg} associates to this the map of group stacks
\[ \SpfBGm \Zp\Brk{\beta, z} / (p, \beta)^n \to \SpfBGm \Zp\Brk{\beta}[z] / (p, \beta)^n \]
with the same group law, so in the colimit we obtain
\[ \SpfBGm \Zp\Brk{\beta, z} \to \SpfBGm \Zp\Brk{\beta}[z]_{(p,\beta)}^\compl \mpunc, \]
which is of the desired form.

It remains to show that this algebraization is that of the conjecture.
By \cite[\S 2.12.2]{drinfeldfg} we may as above check over $\OCp^\Prism$, where by \cite[\S 2.12.5(i)]{drinfeldfg} we need to show that it is given by
\[ \Spf \Prism_{\OCp}\Brk{z} \to \Spf \Prism_{\OCp}[z]\pIcpl \mpunc, \]
with group law given by
\[ z \mapsto z_1 + z_2 + \mu z_1 z_2 \mpunc. \]
But this follows by base change from the computation above.
\end{proof}

\begin{recollection}[The section $s$]
The formal group $\GDr_{\Zp^\Prism}$ admits a certain section $s$ defined in \cite[\S 2.9.1]{drinfeldfg}.
Over the $q$-de Rham prism, this section is given by
\[ \Zp\Brk{q-1, z} \ni z \mapsto [p]_q \in \QdR \mpunc. \]
The algebraization $\GDralg_\QdR$ thus admits a section
\[ \stilde_\QdR : z \mapsto 1 \]
such that $s_\QdR = p \stilde_\QdR$ \cite[\S 2.10.7]{drinfeldfg}.
Here the right-hand side denotes multiplication by $p$ in the abelian group $\GDralg_\QdR$.
\end{recollection}

\begin{lemma}[The section $\stilde$]\label{stildedescent}
If $(A, I)$ is a prism over the $q$-de Rham prism such that $A / (q-1)$ is $p$-torsionfree, then there is a unique section $\stilde_A$ of $\GDralg_A$ such that $s_A = p \stilde_A$.

There is thus a unique section $\stilde$ of $\GDralg_{\Zp^\Prism}$ such that $s = p \stilde$.
\end{lemma} \begin{proof}
Let $(A, I)$ be such a prism, so we may write the coordinate ring of $\GDralg_A$ as
\[ A[z]\pIcpl \]
with group law
\[ z \mapsto z_1 + z_2 + (q-1) z_1 z_2 \mpunc. \]
We need to show that there is a unique $x \in A$ such that $[p](x) = [p]_q$, where $[p](-)$ is the $p$-series of the group law.
Note that this admits the solution $x = 1$, so we need to see that the equation
\[ \frac{[p](x) - [p]_q}{x-1} = 0 \]
admits no solutions.
But modulo $q-1$ this becomes $p = 0$, which admits no solutions by assumption.

The second statement follows by descent from the first applied to the Čech nerve of the cover $\Spf \Prism_{\Zpcyc} \to \Zp^\Prism$.
To apply the first statement, we are using that the terms of the completed conerve of $\Zp \to \Zpcyc$ are $(\zeta_p - 1)$-torsionfree qrsp rings and that the mod-$(p, q-1)$ prismatic cohomology of these terms lies in degree $0$ (being given by the derived de Rham cohomology of their reductions mod $(\zeta_p - 1)$).
\end{proof}

\appendix

\section{Limits of localizations of categories}

\begin{proposition}\label{limlocalizations}
Let $\calI$ be a small category and
\[ \calC_{(-)} : \calI \to \Cathat \]
a functor from $\calI$ to the category of large categories.
Suppose we are given for each $i \in \calI$ a full subcategory $\iota_i : \calD_i \inj \calC_i$ whose inclusion admits a left adjoint $L_i$.
Suppose further that the categories $\calD_i$ are preserved under the functors of $\calC_{(-)}$, i.e.\ that for each morphism $f : i \to j$ in $\calI$ we have $\calC_f(\calD_i) \subseteq \calD_j$.

Write $\calC \defeq \lim_{i \in \calI} \calC_i$, and let $\iota : \calD \to \calC$ be the full subcategory of objects whose image in each $\calC_i$ lands in $\calD_i$.
Then the inclusion $\iota$ admits a left adjoint $L$, and for each $i \in \calI$ the diagram
\[\xymatrix{
\calD \ar[d] \ar[r]^\iota & \calC \ar[d] \\
\calD_i \ar[r]^{\iota_i} & \calC_i
}\]
is left adjointable (in the sense of \cite[Definition 4.7.4.13]{ha}).
\end{proposition} \begin{proof}
Let $\int\calC$ be the cocartesian fibration classified by $\calC_{(-)}$, and let $\int\calD$ be the full subcategory on the objects of the categories $\calD_i$.
The map $\int\calD \to \calI$ is clearly locally cocartesian, so by \cite[Remark 2.4.2.9]{htt} and the assumed preservation of the $\calD_i$ it is cocartesian.
Thus the inclusions $\calD_i \inj \calC_i$ assemble into an $\calI$-indexed functor $\calD_{(-)} \to \calC_{(-)}$.
Note that, as the limit of a diagram of categories is computed as cocartesian sections of its associated fibration, we have $\lim_{i \in \calI} \calD_i \equiv \calD$ compatibly with the maps to $\calC$ and $\calC_{(-)}$.

We thus have a limit diagram
\[ \calI^\triangleleft \to \Fun(\Delta^1, \Cathat) \]
encoding the categories above and the inclusion maps.
By hypothesis the subdiagram
\[ \calI \to \Fun(\Delta^1, \Cathat) \]
factors through the full subcategory of the target on those functors admitting left adjoints, and by the assumed preservation of the subcategories $\calD_i$ the morphisms in $\calI$ are sent to left adjointable diagrams.
In other words, this diagram factors as
\[ \calI \to \Fun^\mathrm{LAd}(\Delta^1, \Cathat) \mpunc, \]
where the target category is defined as in \cite[Definition 4.7.4.16]{ha}.
The result now follows immediately from \cite[Corollary 4.7.4.18]{ha}.
\end{proof}

\begin{proposition}\label{limlocalizationssym}
In the setting of \Cref{limlocalizations}, suppose that we are given a lift of $\calC_{(-)}$ to $\CAlg(\Cathat)$, and that each localization $L_i$ is compatible with the symmetric monoidal structure on $\calC_i$ (in the sense of \cite[Definition 2.2.1.6]{ha}).

Then $L$ is compatible with the induced symmetric monoidal structure on $\calC$.
As such, by \cite[Proposition 2.2.1.9]{ha}, $\calD$ obtains the structure of a symmetric monoidal category such that $L$ is symmetric monoidal and $\iota$ lax symmetric monoidal.

For each $i \in \calI$ the composite
\[ \CAlg(\calD) \to \CAlg(\calC) \to \CAlg(\calC_i) \]
lands in the full subcategory $\CAlg(\calD_i) \inj \CAlg(\calC_i)$, and the induced diagram
\[\xymatrix{
\CAlg(\calD) \ar[d] \ar[r] & \CAlg(\calC) \ar[d] \\
\CAlg(\calD_i) \ar[r] & \CAlg(\calC_i)
}\]
is left adjointable.
\end{proposition} \begin{proof}
The first statement follows from \cite[Example 2.2.1.7]{ha}, and the second is clear from the definitions.
\end{proof}

\begin{definition}\label{nonconncompl}
Let $A$ be an $\Einfty$-ring and $I \subseteq \pi_0 A$ a finitely generated ideal.
Then by \cite[\S 7.3.1]{sag} we have a subcategory
\[ (\Mod_{\tau_{\geq 0} A})_I^\compl \subseteq \Mod_{\tau_{\geq 0} A} \]
whose inclusion admits a symmetric monoidal left adjoint $(-)_I^\compl$.
We define
\[ (\Mod_A)_I^\compl \defeq \Mod_{A_I^\compl}\prn*{(\Mod_{\tau_{\geq 0} A})_I^\compl} \mpunc. \]
This naturally obtains the structure of a symmetric monoidal category, and by the results of \cite[\S 8]{sag} it is functorial in the pair $(A, I)$.
\end{definition}

\begin{construction}\label{loccomplfamily}
Let $\calI$ be a small category,
\[ A_{(-)} : \calI \to \CAlg(\Sp) \]
a family of commutative ring spectra parametrized by $\calI$, $I_{(-)} \subseteq \pi_0 A_{(-)}$ a family of finitely generated ideals, and $\scrS_{(-)} \subseteq \pi_* A_{(-)}$ a family of multiplicative subsets of homogeneous elements.
Suppose that, for each morphism $f : i \to j$ in $\calI$, the associated functor
\[ \Mod_{A_i} \to \Mod_{A_j} \]
sends $\scrS_i$-local modules to $\scrS_j$-local modules.
Then we define as follows a functor
\[ \prn*{(\scrS^{-1} A)_I^\compl}_{(-)} : \calI \to \CAlg(\Sp) \]
whose value at $i \in \calI$ identifies with $(\scrS_i^{-1} A_i)_{I_i}^\compl$.

We have a family of symmetric monoidal categories
\[ \Mod_{A_{(-)}} : \calI \to \CAlg(\Cathat) \]
along with a symmetric monoidal localization
\[ \Mod_{A_i} \to \scrS_i^{-1} \Mod_{A_i} \]
for each $i \in \calI$, where the target is defined as in \cite[\S 7.2.3]{ha}.
By \Cref{limlocalizationssym} we obtain a functor
\[ \prn{\scrS^{-1} \Mod_A}_{(-)} : \calI \to \CAlg(\Cathat) \mpunc. \]
Taking the unit object of the limit of this functor and applying the forgetful functor to $\CAlg(\Sp^\calI)$ yields a family of commutative ring spectra
\[ \prn{\scrS^{-1} A}_{(-)} : \calI \to \CAlg(\Sp) \mpunc. \]
We also obtain from $I_{(-)}$ a family of finitely generated ideals
\[ \prn{\scrS^{-1} I}_{(-)} \subseteq \pi_0 \prn{\scrS^{-1} A}_{(-)} \mpunc, \]
so \Cref{nonconncompl} yields a functor
\[ \prn*{(\scrS^{-1} \Mod_A)_I^\compl}_{(-)} : \calI \to \CAlg(\Cathat) \mpunc. \]
We obtain the desired object by once again taking the unit object of the limit and forgetting to $\CAlg(\Sp^\calI)$.
\end{construction}

\printbibliography

@misc{ellcoh1,
  title={Elliptic Cohomology I: Spectral Abelian Varieties},
  author={Lurie, Jacob},
  year={2018},
  url={https://www.math.ias.edu/~lurie/papers/Elliptic-I.pdf},
  urldate={2024-02-01},
}

@misc{drinfeldfg,
  title={A 1-dimensional formal group over the prismatization of {$\Spf\bbZ_p$}},
  author={Drinfeld, Vladimir},
  year={2023},
  eprint={2107.11466},
  archivePrefix={arXiv},
  primaryClass={math.AG},
}

@article{bms1,
  title={Integral $p$-adic Hodge theory},
  author={Bhatt, Bhargav and Morrow, Matthew and Scholze, Peter},
  journal={Publications math{\'e}matiques de l'IH{\'E}S},
  volume={128},
  number={1},
  pages={219--397},
  year={2018},
  publisher={Springer},
}

@article{bms2,
  title={Topological Hochschild homology and integral $p$-adic Hodge theory},
  author={Bhatt, Bhargav and Morrow, Matthew and Scholze, Peter},
  journal={Publications math{\'e}matiques de l'IH{\'E}S},
  volume={129},
  number={1},
  pages={199--310},
  year={2019},
  publisher={Springer},
}

@article{prisms,
  title = {Prisms and prismatic cohomology},
  author={Bhatt, Bhargav and Scholze, Peter},
  volume = {196},
  journal = {Annals of Mathematics},
  number = {3},
  publisher = {Department of Mathematics of Princeton University},
  pages = {1135--1275},
  year = {2022},
  doi = {10.4007/annals.2022.196.3.5},
}

@book{htt,
  title={Higher Topos Theory},
  author={Lurie, Jacob},
  year={2017},
% year={2009},
% publisher={Princeton University Press},
  url={https://www.math.ias.edu/~lurie/papers/HTT.pdf},
  urldate={2024-02-01},
}

@book{ha,
  title={Higher Algebra},
  author={Lurie, Jacob},
  year={2017},
  url={https://www.math.ias.edu/~lurie/papers/HA.pdf},
  urldate={2024-02-01},
}

@misc{dag8,
  title = {Derived Algebraic Geometry VIII: Quasi-Coherent Sheaves and Tannaka Duality Theorems},
  author = {Jacob Lurie},
  year = {2011},
  url = {https://www.math.ias.edu/~lurie/papers/DAG-VIII.pdf},
  urldate = {2024-02-01},
}

@misc{moravathh,
  title = {Complex orientations for THH of some perfectoid fields},
  author = {Jack Morava},
  year = {2020},
  eprint = {1608.04702},
  archivePrefix = {arXiv},
  primaryClass = {math.AT},
}

@article{orientationoftp,
  shorthand={AK22},
  title={Complex orientations and $\text{TP}$ of complete DVRs},
  author={Gabriel Angelini-Knoll},
  journal={Homology, Homotopy \& Applications},
  volume={25},
  number={1},
  year={2023},
  doi={10.4310/HHA.2023.v25.n1.a16},
}

@misc{prismatization,
  title={Prismatization},
  author={Vladimir Drinfeld},
  year={2022},
  eprint={2005.04746},
  archivePrefix={arXiv},
  primaryClass={math.AG},
}

@misc{blprismatization,
  title={The prismatization of $p$-adic formal schemes},
  author={Bhargav Bhatt and Jacob Lurie},
  year={2022},
  eprint={2201.06124},
  archivePrefix={arXiv},
  primaryClass={math.AG},
}

@misc{apc,
  title={Absolute prismatic cohomology},
  author={Bhargav Bhatt and Jacob Lurie},
  year={2022},
  eprint={2201.06120},
  archivePrefix={arXiv},
  primaryClass={math.AG},
}

@misc{fgauges,
  title={Prismatic $F$-gauges},
  author={Bhatt, Bhargav},
  year={2022},
  url={https://www.math.ias.edu/~bhatt/teaching/mat549f22/lectures.pdf},
  urldate={2024-02-01},
}

@incollection{hntc,
  title={Topological cyclic homology},
  author={Hesselholt, Lars and Nikolaus, Thomas},
  booktitle={Handbook of Homotopy Theory},
  editor={Haynes Miller},
  year={2019},
  publisher={Chapman and Hall/CRC},
  chapter={15},
  pages={619--656},
}

@misc{mfgorconn,
  title={Moduli stack of oriented formal groups and cellular motivic spectra over $\mathbf C$},
  author={Rok Gregoric},
  year={2021},
  eprint={2111.15212},
  archivePrefix={arXiv},
  primaryClass={math.AT},
}

@misc{sag,
  title={Spectral Algebraic Geometry},
  author={Jacob Lurie},
  url={https://www.math.ias.edu/~lurie/papers/SAG-rootfile.pdf},
  year={2018},
  urldate={2024-02-01},
}

@misc{unipotent,
  title={Unipotent homotopy theory of schemes},
  author={Shubhodip Mondal and Emanuel Reinecke},
  year={2023},
  eprint={2302.10703},
  archivePrefix={arXiv},
  primaryClass={math.AG}
}

@misc{frobheight,
  title={Frobenius height of prismatic cohomology with coefficients},
  author={Haoyang Guo and Shizhang Li},
  year={2023},
  eprint={2309.06663},
  archivePrefix={arXiv},
  primaryClass={math.AG},
}

@article{bmsyntomic,
  title={Syntomic complexes and p-adic étale Tate twists},
  volume={11},
  DOI={10.1017/fmp.2022.21},
  journal={Forum of Mathematics, Pi},
  publisher={Cambridge University Press},
  author={Bhatt, Bhargav and Mathew, Akhil},
  year={2023},
  pages={e1}
}

@misc{cmaic,
  title={Cohen-Macaulayness of absolute integral closures},
  author={Bhargav Bhatt},
  year={2021},
  eprint={2008.08070},
  archivePrefix={arXiv},
  primaryClass={math.AG},
}

@misc{dhry,
  title = {Prismatization of ring spectra},
  author = {Sanath Devalapurkar and Jeremy Hahn and Arpon Raksit and Allen Yuan},
  pubstate = {inpreparation},
}

@article{fcrystalsGaloisreps,
  title={Prismatic $F$-crystals and crystalline Galois representations},
  author={Bhatt, Bhargav and Scholze, Peter},
  journal={Cambridge Journal of Mathematics},
  volume={11},
  number={2},
  pages={507--562},
  year={2023},
  publisher={International Press of Boston},
  doi={10.4310/CJM.2023.v11.n2.a3},
}

@article{wufcrys,
  title={Galois representations, $(\phi, \Gamma)$-modules and prismatic F-crystals},
  author={Wu, Zhiyou},
  journal={Documenta Mathematica},
  volume={26},
  pages={1771--1798},
  year={2021},
  doi={10.4171/DM/855},
}

@book{gr1,
  title={A study in derived algebraic geometry: Volume I: correspondences and duality},
  author={Gaitsgory, Dennis and Rozenblyum, Nick},
  volume={221},
  year={2019},
  publisher={American Mathematical Society}
}

@article{bhv,
  title = {Local duality in algebra and topology},
  author = {Tobias Barthel and Drew Heard and Gabriel Valenzuela},
  journal = {Advances in Mathematics},
  volume = {335},
  pages = {563--663},
  year = {2018},
  issn = {0001-8708},
  doi = {10.1016/j.aim.2018.07.017},
}

@inproceedings{hoveystable,
  title = {Homotopy theory of comodules over a Hopf algebroid},
  author = {Mark Hovey},
  year = {2004},
  booktitle = {Homotopy Theory: Relations with Algebraic Geometry, Group Cohomology, and Algebraic $K$-Theory},
  publisher = {American Mathematical Society},
  address = {Providence, Rhode Island},
  series = {Contemporary Mathematics},
  volume = {346},
  pages = {261--304},
  editor = {Paul Goerss and Stewart Priddy},
  doi = {10.1090/conm/346/06291},
}

@article{synthetic,
  title = {Synthetic spectra and the cellular motivic category},
  author = {Pstr{\k{a}}gowski, Piotr},
  journal = {Inventiones mathematicae},
  volume = {232},
  number = {2},
  pages = {553--681},
  year = {2023},
  publisher = {Springer},
  doi = {10.1007/s00222-022-01173-2},
}

@misc{mmaffstk,
  title = {Affine stacks and derived rings},
  author = {Akhil Mathew and Shubhodip Mondal},
  pubstate = {inpreparation},
}

@article{naumannmfg,
  title = {The stack of formal groups in stable homotopy theory},
  author = {Niko Naumann},
  journal = {Advances in Mathematics},
  volume = {215},
  number = {2},
  pages = {569--600},
  year = {2007},
  issn = {0001-8708},
  doi = {10.1016/j.aim.2007.04.007},
}

@online{f1syn,
  title = {On the prismatization of the field with one element},
  author = {Jacob Lurie},
  url = {https://www.youtube.com/watch?v=Kib-oa7YfSw},
  year = {2024},
  urldate = {2024-01-24},
}

@online{analyticstacks,
  title = {Analytic Stacks},
  author = {Dustin Clausen and Peter Scholze},
  url = {https://www.youtube.com/playlist?list=PLx5f8IelFRgGmu6gmL-Kf_Rl_6Mm7juZO},
  year = {2024},
  urldate = {2024-02-01},
}

@misc{stefanichtannaka,
      title={Tannaka duality and 1-affineness},
      author={Germán Stefanich},
      year={2023},
      eprint={2311.04515},
      archivePrefix={arXiv},
      primaryClass={math.AG},
}

@article{champsaffines,
  author="Toën, Bertrand",
  title="Champs affines",
  journal="Selecta Mathematica",
  year="2006",
  month="Jun",
  day="01",
  volume="12",
  number="1",
  pages="39--134",
  issn="1420-9020",
  doi="10.1007/s00029-006-0019-z",
}

@misc{arponhh,
  title={Hochschild homology and the derived de Rham complex revisited},
  author={Arpon Raksit},
  year={2020},
  eprint={2007.02576},
  archivePrefix={arXiv},
  primaryClass={math.AG},
}

@misc{tangsyntomic,
  title={Syntomic cycle classes and prismatic Poincar\'e duality},
  author={Longke Tang},
  year={2022},
  eprint={2210.14279},
  archivePrefix={arXiv},
  primaryClass={math.AG},
}

@article{albtrace,
  shorthand = {ALB20},
  author = {Johannes Anschütz and Arthur-César {Le Bras}},
  title = {{The $p$-completed cyclotomic trace in degree $2$}},
  volume = {5},
  journal = {Annals of K-Theory},
  number = {3},
  publisher = {MSP},
  pages = {539--580},
  year = {2020},
  doi = {10.2140/akt.2020.5.539},
}

@article{cofibtau,
  author = {Bogdan Gheorghe and Guozhen Wang and Zhouli Xu},
  title = {{The special fiber of the motivic deformation of the stable homotopy category is algebraic}},
  volume = {226},
  journal = {Acta Mathematica},
  number = {2},
  publisher = {Institut Mittag-Leffler},
  pages = {319--407},
  year = {2021},
  doi = {10.4310/ACTA.2021.v226.n2.a2},
}

@article{Cmmf,
  title={$\mathbb{C}$-motivic modular forms},
  author={Gheorghe, Bogdan and Isaksen, Daniel C and Krause, Achim and Ricka, Nicolas},
  journal={Journal of the European Mathematical Society},
  volume={24},
  number={10},
  pages={3597--3628},
  year={2021}
}

@article{quillenmu,
  author = {Daniel Quillen},
  title = {On the formal group laws of unoriented and complex cobordism theory},
  volume = {75},
  journal = {Bulletin of the American Mathematical Society},
  number = {6},
  publisher = {American Mathematical Society},
  pages = {1293--1298},
  year = {1969},
}

@article{tasosa1modgm,
  author = {Moulinos, Tasos},
  title = {The geometry of filtrations},
  journal = {Bulletin of the London Mathematical Society},
  volume = {53},
  number = {5},
  pages = {1486--1499},
  doi = {10.1112/blms.12512},
  year = {2021},
}

@phdthesis{dag,
  title={Derived algebraic geometry},
  author={Lurie, Jacob},
  year={2004},
  school={Massachusetts Institute of Technology}
}

@article{ravenelXn,
  title = {Localization with respect to certain periodic homology theories},
  author = {Douglas C.\ Ravenel},
  journal = {American Journal of Mathematics},
  volume = {106},
  number = {2},
  pages = {351--414},
  year = {1984},
  doi = {10.2307/2374308},
}

@misc{goerssQCohMfg,
  title = {Quasi-coherent sheaves on the moduli stack of formal groups},
  author = {Paul G. Goerss},
  year = {2008},
  eprint = {0802.0996},
  archivePrefix = {arXiv},
  primaryClass = {math.AT},
}

\end{document}